\documentclass{conm-p-l}

\usepackage{amssymb}
\def\P{\mathbb P}
\newcommand{\R}{\mathbb R}

\newcommand{\Z}{\mathbb Z}
\newcommand{\Q}{\mathbb Q}

\newcommand{\F}{\mathbb F}
\newcommand{\T}{\mathbb T}
\newcommand{\V}{\mathbb V}
\newcommand{\M}{\mathbb M}
\newcommand{\st}[1]{\vskip 1mm \noindent\makebox[4mm][r]{\bf #1} \hspace{1mm}}
\newcommand{\stst}[1]{\vskip 1mm \noindent\makebox[4mm][r]{\bf #1} \hspace{6mm}}
\newcommand{\ststst}[1]{\vskip 1mm \noindent\makebox[4mm][r]{\bf #1} \hspace{11mm}}
\newcommand{\stststst}[1]{\vskip 1mm \noindent\makebox[4mm][r]{\bf #1} \hspace{16mm}}
\newcommand{\ststststst}[1]{\vskip 1mm \noindent\makebox[4mm][r]{\bf #1} \hspace{21mm}}

\def\op{\operatorname}

\def\as#1{\renewcommand\arraystretch{#1}}

\def\cs{\operatorname{cs}}

\def\dsc{\operatorname{Disc}}
\def\diso{\lower.4ex\hbox{$\downarrow$}\raise.4ex\hbox{\mbox{\scriptsize
$\wr$}}}

\def\ep{\epsilon}
\def\exp{\op{exp}}

\def\ga{\gamma}

\def\gg{\mathcal{G}r}
\def\ggm{\mathcal{G}r(\mu)}

\def\hcs{h_{\cs}}

\def\hm{\op{H}_\mu}

\def\imp{\,\Longrightarrow\,}

\def\ind{\op{ind}}
\def\iso{\ \lower.3ex\hbox{\as{.08}$\begin{array}{c}\lra\\\mbox{\tiny $\sim\,$}\end{array}$}\ }

\def\kb{\overline{K}_v}
\def\kp{\op{KP}}
\def\kpm{\op{KP}(\mu)}

\def\la{\lambda}
\def\lg{l\raise.6ex\hbox to.2em{\hss.\hss}l}
\def\ll{\mathcal{L}}

\def\lra{\longrightarrow}
\def\m{{\mathfrak m}}
\def\md#1{\; \left(\op{mod}\ {#1}\right)}

\def\mmu{\mid_\mu}
\def\mn{\op{Min}}
\def\mx{\op{Max}}

\def\nph{N_{\mu,\phi}}

\def\om{\omega}
\def\oo{\mathcal{O}}
\def\orb{\hbox to  .3em{$\backslash$}\backslash}
\def\ord{\op{ord}}
\def\p{\mathfrak{p}}

\def\ppa{\mathcal{P}_{\alpha}}
\def\pset{\mathcal{P}}

\def\res{\operatorname{Res}}

\def\rr{\mathcal{R}}
\def\rrm{\mathcal{R}_\mu}

\def\smu{\sim_\mu}

\def\t{\theta}

\def\tcal{\mathcal{T}}
\def\tst{\tcal^{\operatorname{str}}}
\def\ty{\mathbf{t}}

\def\Vi{\V^{\op{ind}}}

\newtheorem{theorem}{Theorem}[section]
\newtheorem{lemma}[theorem]{Lemma}
\newtheorem{corollary}[theorem]{Corollary}
\newtheorem{proposition}[theorem]{Proposition}

\theoremstyle{definition}
\newtheorem{definition}[theorem]{Definition}
\newtheorem{example}[theorem]{Example}

\theoremstyle{remark}
\newtheorem{remark}[theorem]{Remark}

\numberwithin{equation}{section}

\begin{document}
\title{Genetics of polynomials over local fields}

\author[Gu\`ardia]{Jordi Gu\`ardia}
\address{Departament de Matem\`atica Aplicada IV, Escola Polit\`ecnica Superior d'Enginye\-ria de Vilanova i la Geltr\'u, Av. V\'\i ctor Balaguer s/n. E-08800 Vilanova i la Geltr\'u, Catalonia}
\email{guardia@ma4.upc.edu}

\author[Nart]{Enric Nart}
\address{Departament de Matem\`{a}tiques,
         Universitat Aut\`{o}noma de Barcelona,
         Edifici C, E-08193 Bellaterra, Barcelona, Catalonia}
\email{nart@mat.uab.cat}
\thanks{Partially supported by MTM2012-34611 from the Spanish MEC}
\date{}

\makeatletter
\@namedef{subjclassname@2010}{%
  \textup{2010} Mathematics Subject Classification}







\begin{abstract}
Let $(K,v)$ be a discrete valued field with valuation ring  $\oo$, and let $\oo_v$ be the completion of $\oo$ with respect to the $v$-adic topology.
In this paper we discuss the advantages of manipulating polynomials in $\oo_v[x]$ on a computer by means of OM representations of prime (monic and irreducible) polynomials. An OM representation supports discrete data characterizing the Okutsu equivalence class of the prime polynomial. These discrete parameters are a kind of DNA sequence common to all individuals in the same Okutsu class, and they contain relevant arithmetic information about the polynomial and the extension of $K_v$ that it determines.
\end{abstract}

\maketitle

\section*{Introduction}
Polynomials with $p$-adic coefficients arose in a purely algebraic context with Hensel's reinterpretation of the ideas of Kummer and Dedekind about factorization of algebraic integers. The prime polynomials in $\Z_p[x]$, whose roots in $\overline{\Q}_p$ are algebraic over $\Q$, parameterize  prime ideals dividing the prime number $p$ in maximal orders of number fields.

More generally, let $A$ be a Dedekind domain with field of fractions $K$, let $f\in A[x]$ be a monic irreducible separable polynomial of degree $n$ and let $L=K[x]/(f)$ be the finite extension of $K$ determined by $f$. The prime ideals of the integral closure of $A$ in $L$ dividing a given prime ideal $\p$ in $A$ are in 1-1 correspondence with the prime factors of $f$ in $\hat{A}_\p[x]$, where $\hat{A}_\p$ is the completion of $A$ with respect to the $\p$-adic topology. This leads to a wide scope of arithmetic problems where prime polynomials over local fields play a significant role. For instance, the analysis of the ramification of a finite separable morphism between two algebraic curves is one of such problems.

In this paper, we deal with an arbitrary discrete valued field $(K,v)$ with valuation ring $\oo$. Let $K_v$ be the completion of $K$ at $v$ and denote by $\oo_v\subset K_v$ the valuation ring of $K_v$. Given a monic square-free polynomial $f\in\oo[x]$, we are interested in the computation of the prime factors of $f$ in $\oo_v[x]$. 

From a computational perspective, polynomials in $\oo_v[x]$ are approximated by polynomials in $\oo[x]$, following closely  the paradigm of polynomials with real coefficients. However, prime polynomials in $\oo_v[x]$ are much richer objects because they have an algebraic substrate containing relevant arithmetic information.
This substrate is described by a sequence of discrete parameters, which are a kind of DNA sequence common to all prime polynomials which are sufficiently close one to each other.
Thanks to their discrete nature, these genetic data admit an exact computation.

The aim of this paper is to give a precise description of this genetic information, and to explain how it may be computed. To this end, we present under a new framework some previous work of several authors, on computational tools for $v$-adic factorization.

Consider the simplest extension of $v$ to a valuation $\mu_0$ on $K_v(x)$; that is, $\mu_0$ acts on polynomials as
$$
\mu_0\left(a_0+a_1x+\cdots+a_tx^t\right)=\mn\left\{v(a_0),\dots,v(a_t)\right\}.
$$
Let $F\in\oo_v[x]$ be a prime polynomial and $\t\in\overline{K}_v$ one of its roots. The valuation $v$ admits a unique extension to $\overline{K}_v$ and we may consider the pseudo-valuation $\mu_{\infty,F}$ on $K_v[x]$ defined as $\mu_{\infty,F}(g)=v(g(\t))$ for any $g\in K_v[x]$. 


In a pioneering paper, MacLane described an inductive structure on the set of all discrete valuations on $K(x)$ extending $v$ \cite{mcla,mclb}. He expressed the pseudo-valuation $\mu_{\infty,F}$ as a limit of such valuations and showed that in this approximation process there is a finite chain of valuations:
$$
\mu_0 < \mu_1<\dots <\mu_r <\mu_{\infty,F}
$$
intrinsically attached to $F$. Let us denote $\mu_F:=\mu_r$. Most of the genetic information of $F$ is provided by certain invariants and operators attached to these valuations $\mu_0,\dots,\mu_r$. 

In 2007, Vaqui\'e reviewed and generalized MacLane's work to arbitrary valued fields $(K,v)$ which are not necessarily discrete \cite{Vaq,Vaq2,Vaq3}. The use of the graded algebra $\ggm$ attached to a valuation $\mu$ on $K(x)$ led Vaqui\'e to a more elegant presentation of the theory. A key role is played by the  
\emph{residual ideals} in the degree-zero subring $\Delta(\mu)$ of $\ggm$. The residual ideal of a polynomial $g\in K[x]$ is defined as $\rr_\mu(g)=H_\mu(g)\ggm\cap\Delta(\mu)$, where $\hm(g)$ is the image of $g$ in the piece of degree $\mu(g)$ of the algebra.

In a recent paper \cite{Rideals}, this approach of Vaqui\'e was extended with a constructive treatment of the theory in the discrete case. On the set $\P$ of all prime polynomials in $\oo_v[x]$, the following equivalence relation is considered in \cite{okutsu}: two prime polynomials $F,G\in\P$ of the same degree are \emph{Okutsu equivalent}, and we write $F\approx G$, if the \emph{quality} of $G$ as an approximation to $F$ is greater than certain \emph{Okutsu bound} $\delta_0(F)$.  
The main result of \cite{Rideals} establishes a canonical bijection between the quotient set $\,\P/\!\approx\,$ and the \emph{MacLane space} $\M$, defined as the set of all pairs $(\mu,\ll)$, where $\mu$ is an inductive valuation on $K(x)$ and $\ll$ is a \emph{strong} maximal ideal of $\Delta(\mu)$. The bijection sends the class of $F$ to the pair 
$(\mu_F,\rr_{\mu_F}(F))$. 

The point $(\mu,\ll)$ of the MacLane space which corresponds to the Okutsu class of a prime polynomal $F\in\P$ is, by definition, the \emph{genetic code} of $F$.  Thus, two prime polynomials have the same genetic code if and only if they are Okutsu equivalent.

Let us now outline the content of the paper, which is a natural continuation of \cite{Rideals}. In sections 1 and 2, we sketch the results of \cite{Rideals} in order to collect all technical definitions and results which are needed for the rest of the paper. In sections 3 and 4 we discuss \emph{types} and \emph{OM representations} as the computational objects which are able to support the genetic data of polynomials. The initials OM stand for Ore-MacLane or Okutsu-Montes indistinctly. We parameterize the MacLane space $\M$ by certain set $\T$ of equivalence classes of types and we introduce the \emph{genomic tree} of a square-free polynomial $f\in \oo[x]$ as a discrete object gathering the genetic information of all prime factors of $f$.
In section 5 we present an adapted version of the Montes algorithm, aiming at the computation of the genomic tree of $f$, together with an  approximation to each prime factor by an Okutsu equivalent polynomial in $\oo[x]$. The knowledge of the genetic code of a prime polynomial facilitates the resolution of many computational tasks concerning this polynomial. Section 6 is devoted to the discussion of these algorithmic applications. 

For a monic, irreducible and separable $f\in\oo[x]$, let $L$ be the finite extension of $K$ determined by $f$. It is well known that the computation of sufficiently good approximations to the prime factors of $f$ in $\oo_v[x]$ leads to the design of routines for the computation of the integral closure of $\oo$ in $L$, the $v$-part of the discriminant of $L/K$, and the resolution of similar arithmetic tasks concerning the extension $L/K$.  
Thus, the use of the Montes algorithm as a fast method to compute approximate $v$-adic factorizations leads to an improvement of many classical arithmetic algorithms. But this is not the spirit of the routines of section 6. For each routine, we find a tight link between some  arithmetic problem and the genetics of certain prime polynomials. This leads to an original design for the routine and to a much better practical performance.  

The concept of a type and the Montes algorithm were introduced in \cite{montes} for $v$ a discrete valuation on a global field $K$. These results were reviewed in \cite{algorithm,HN} and their computational implications were developed in a series of papers \cite{BNS,newapp, bases, GNP, Ndiff}. The derivation of these tools from the modern presentation of MacLane's valuations in the spirit of Vaqui\'e, leads to a more elegant treatment of the subject and to its generalization to arbitrary discrete valued fields $(K,v)$.

\section{MacLane valuations} \label{secMacLane}

Let $K$ be a field equipped with a discrete valuation $v\colon K^*\to \Z$, normalized so that $v(K^*)=\Z$. Let $\oo$ be the valuation ring of $K$, $\m$ the maximal ideal, $\pi\in\m$ a generator of $\m$ and $\F=\oo/\m$ the residue class field. 

Let $K_v$ be the completion of $K$ and denote still by $v\colon \kb^*\to \Q$ the canonical extension of $v$ to a fixed algebraic closure of $K_v$. Let $\oo_v$ be the valuation ring of $K_v$, $\m_v$ its maximal ideal and $\F_v=\oo_v/\m_v$ the residue class field. We consider the canonical isomorphism $\F\simeq\F_v$ as an identity and we indicate simply with a bar, $\raise.8ex\hbox{---}\colon \oo_v[x]\longrightarrow \F[x]$,
the homomorphism of reduction of polynomials modulo $\m_v$. 

Let $\V$ be the set of discrete valuations, $\mu\colon K(x)^*\to \Q$, such that $\mu_{\mid K}=v$ and $\mu(x)\ge0$. From now on, the elements of $\V$ will be simply called \emph{valuations}.

For any valuation $\mu\in\V$, we denote

\begin{itemize}
\item $\Gamma(\mu)=\mu\left(K(x)^*\right)\subset \Q$, the cyclic group of finite values of $\mu$. 
\item $e(\mu)>0$, the ramification index of $\mu$, determined by $\Gamma(\mu)=e(\mu)^{-1}\Z$.   
\end{itemize}

In the set $\V$ there is a natural partial ordering:
$$
\mu\le \mu' \quad\mbox{ if }\quad\mu(g) \le \mu'(g), \ \forall\,g\in K[x]. 
$$

We denote by $\mu_0\in \V$ the valuation which acts on polynomials as
$$
\mu_0\left(a_0+a_1x+\cdots+a_tx^t\right)=\mn_{0\le s\le t}\left\{v(a_s)\right\}.
$$
Clearly, $\mu_0\le \mu$ for all $\mu\in\V$; in other words, $\mu_0$ is the minimum element in $\V$.

In this section we describe a certain subset $\Vi\subset\V$ introduced by MacLane \cite{mcla}, formed by the so-called \emph{inductive valuations}. The modern presentation of this topic in the language of graded algebras is due to Vaqui\'e \cite{Vaq}. We follow the development of \cite{Rideals} which included a constructive treatment of the subject.

\subsection{Key polynomials and augmented valuations}\label{subsecKeyP}
Let $\mu\in\V$ be a valuation. For any $\alpha\in\Gamma(\mu)$ we consider the following $\oo$-submodules in $K[x]$:
$$
\ppa=\ppa(\mu)=\{g\in K[x]\mid \mu(g)\ge \alpha\}\supset
\ppa^+=\ppa^+(\mu)=\{g\in K[x]\mid \mu(g)> \alpha\}.
$$    

The \emph{graded algebra of $\mu$} is the integral domain:
$$
\ggm:=\bigoplus\nolimits_{\alpha\in\Gamma(\mu)}\ppa/\ppa^+.
$$

Let $\;\Delta(\mu)=\pset_0/\pset_0^+$ be the piece of degree zero of this algebra. Clearly, $\oo\subset\pset_0$ and $\m=\pset_0^+\cap \oo$; thus, there is a canonical homomorphism $\F\to\Delta(\mu)$ equipping  $\Delta(\mu)$ (and $\ggm$) with a canonical structure of $\F$-algebra.

There is a natural map $\hm\colon K[x]\lra \ggm$, given by $\hm(0)=0$, and
$$\hm(g)= g+\pset_{\mu(g)}^+\in\pset_{\mu(g)}/\pset_{\mu(g)}^+, \mbox{ for }g\ne0.$$
This map does not respect addition but it is multiplicative: $\hm(gh)=\hm(g)\hm(h)$ for all $g,h\in K[x]$.

If $\mu\le \mu'$, then a canonical homomorphism of graded algebras $\ggm\to\gg(\mu')$ is determined by $g+\ppa^+(\mu)\mapsto g+\ppa^+(\mu')$ for all $g,\alpha$. Clearly, $\hm(g)$ belongs to $\op{Ker}(\ggm\to\gg(\mu'))$ if and only if $\mu(g)<\mu'(g)$. 
 
\begin{definition}
Let $g,h,\phi\in K[x]$. We say that:

$g,h$ are \emph{$\mu$-equivalent}, and we write $g\smu h$, if $\hm(g)=\hm(h)$.  

$g$ is \emph{$\mu$-divisible} by $h$, and we write $h\mmu g$, if $\hm(h)\mid\hm(g)$ in $\ggm$.

$\phi$ is \emph{$\mu$-irreducible} if $\hm(\phi)\ggm$ is a non-zero prime ideal.

$\phi$ is \emph{$\mu$-minimal} if $\deg \phi>0$ and $\phi\nmid_\mu g$ for any non-zero $g$ with $\deg g<\deg \phi$.

A \emph{key polynomial} for $\mu$ is a monic polynomial $\phi\in K[x]$ which is $\mu$-minimal and $\mu$-irreducible. We denote by $\kpm$ the set of all key polynomials for $\mu$. 
\end{definition}

For instance, $\kp(\mu_0)$ is the set of all monic polynomials $g\in\oo[x]$ such that $\overline{g}$ is irreducible in $\F[x]$.

\begin{lemma}\label{irredKv}
Every $\phi\in \kpm$ is irreducible in $K_v[x]$ and it belongs to $\oo[x]$.
\end{lemma}

Take $\phi\in \kpm$ and $\nu\in \Q_{>0}$. The \emph{augmented valuation} of $\mu$ with respect to these data is the valuation $\mu'$ determined by the following action on $K[x]$:
$$\mu'(g)=\mn\nolimits_{0\le s}\{\mu(a_s\phi^s)+s\nu\},$$ 
where $g=\sum_{0\le s}a_s\phi^s$ is the canonical $\phi$-expansion of $g$. We denote $\mu'=[\mu;\phi,\nu]$.

\begin{proposition}\label{extension}\mbox{\null}
\begin{enumerate}
\item 
The natural extension of $\mu'$ to $K(x)$ is a valuation on this field and $\mu\le\mu'$.
\item
$\op{Ker}\left(\ggm\to\gg(\mu')\right)=\hm(\phi)\,\ggm$.
\item $\phi$ is a key polynomial for $\mu'$ too.
\end{enumerate}
\end{proposition}

Denote $\Delta=\Delta(\mu)$, and let $I(\Delta)$ be the set of ideals in $\Delta$. Consider the following \emph{residual ideal operator}, which translates questions about $K[x]$ and $\mu$ into ideal-theoretic considerations in the ring $\Delta$:
$$
\rr=\rrm\colon K[x]\lra I(\Delta),\qquad g\mapsto \Delta\cap \hm(g)\ggm.
$$

Let $\phi$ be a key polynomial for $\mu$. Choose a root $\t \in\kb$ of $\phi$ and denote $K_\phi=K_v(\t)$ the finite extension of $K_v$ generated by $\t$. Also, let $\oo_\phi\subset K_\phi$ be the valuation ring of $K_\phi$, $\m_\phi$ the maximal ideal and $\F_\phi=\oo_\phi/\m_\phi$ the residue class field. 

\begin{proposition}\label{sameideal}
If $\phi$ is a key polynomial for $\mu$, then 
\begin{enumerate}
\item $\rr(\phi)=\op{Ker}(\Delta\twoheadrightarrow \F_\phi)$ for the onto homomorphism $\Delta\to \F_\phi$ determined by $g+\pset^+_0\ \mapsto\ g(\t)+\m_\phi$. 
In particular, $\rr(\phi)$ is a maximal ideal of $\Delta$.
\item $\rr(\phi)=\op{Ker}(\Delta\to \Delta(\mu'))$ for any augmented valuation  $\mu'=[\mu;\phi,\nu]$. Thus, the image of $\Delta\to\Delta(\mu')$ is a field canonically isomorphic to $\F_\phi$. 
\end{enumerate}
\end{proposition}

\subsection{Newton polygons}\label{subsecNewton}
The choice of a key polynomial $\phi$ for a valuation $\mu$ determines a \emph{Newton polygon operator}
$$
\nph\colon K[x]\lra 2^{\R^2},
$$
where $2^{\R^2}$ is the set of subsets of the euclidean plane $\R^2$. The Newton polygon of the zero polynomial is the empty set. If $g=\sum_{0\le s}a_s\phi^s$ is the canonical $\phi$-expansion of a non-zero polynomial $g\in K[x]$, then $\nph(g)$ is the lower convex hull of the cloud of points $(s,\mu(a_s\phi^s))$ for all $0\le s$. Figure \ref{figNmodel} shows the typical shape of $\nph(g)$.

If the Newton polygon $N=\nph(g)$ is not a single point, we formally write $N=S_1+\cdots +S_k$, where $S_i$ are the sides of $N$, ordered by their increasing slopes. The left and right end points of $N$ and the points joining two sides of different slopes are called the \emph{vertices} of $N$.

Usually, we shall be interested only in the \emph{principal Newton polygon} $\nph^-(g)$ formed by the sides of negative slope. If there are no sides of negative slope, then $\nph^-(g)$ is the left end point of $\nph(g)$.

The \emph{length} $\ell(N)$ of a Newton polygon $N$ is the abscissa of its right end point.

\begin{lemma}\label{length2}
For every non-zero polynomial $g\in K[x]$, we have
$$\ell(\nph^-(g))=\ord_{\mu,\phi}(g),$$ where $\ord_{\mu,\phi}(g)$ denotes the largest integer $s$ such that $\phi^s\mid_\mu g$. 
\end{lemma}

\begin{figure}
\caption{Newton polygon of a polynomial $g\in K[x]$}\label{figNmodel}
\begin{center}
\setlength{\unitlength}{5mm}
\begin{picture}(14,9)
\put(10.8,2.8){$\bullet$}\put(8.8,3.8){$\bullet$}\put(7.8,2.8){$\bullet$}\put(5.8,4.8){$\bullet$}
\put(4.8,3.8){$\bullet$}\put(3.8,5.8){$\bullet$}\put(2.8,7.8){$\bullet$}
\put(-1,1){\line(1,0){15}}\put(0,0){\line(0,1){9}}
\put(8,3){\line(-3,1){3}}\put(5,4){\line(-1,2){2}}\put(8,3.03){\line(-3,1){3}}
\put(5,4.03){\line(-1,2){2}}\put(11,3){\line(-1,0){3}}\put(11,3.02){\line(-1,0){3}}
\put(12.85,4.8){$\bullet$}\put(11,3){\line(1,1){2}}\put(11,3.02){\line(1,1){2}}
\multiput(3,.9)(0,.25){29}{\vrule height2pt}
\multiput(8,.9)(0,.25){9}{\vrule height2pt}
\multiput(13,.9)(0,.25){16}{\vrule height2pt}
\put(7,.1){\begin{footnotesize}$\ord_{\mu,\phi}(g)$\end{footnotesize}}
\put(2.1,.15){\begin{footnotesize}$\ord_{\phi}(g)$\end{footnotesize}}
\put(12,.15){\begin{footnotesize}$\ell(\nph(g))$\end{footnotesize}}
\multiput(-.1,3)(.25,0){55}{\hbox to 2pt{\hrulefill }}
\put(6,8){\begin{footnotesize}$\nph(g)$\end{footnotesize}}
\put(-1.7,2.85){\begin{footnotesize}$\mu(g)$\end{footnotesize}}
\put(-.45,.35){\begin{footnotesize}$0$\end{footnotesize}}
\end{picture}
\end{center}
\end{figure}
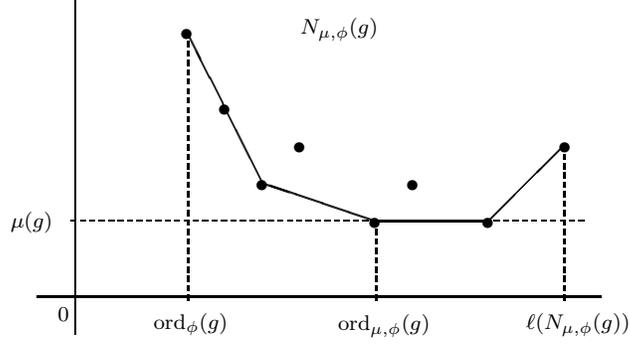

Let $\nu$ be a positive rational number and let $L_{-\nu}$ be the line of slope $-\nu$ which first touches the polygon $N_{\mu,\phi}(g)$ from below.

We define the \emph{$\nu$-component} of $N=N_{\mu,\phi}(g)$ as the segment 
$$S_\nu(g):=\{(x,y)\in N\mid y+\nu x\mbox{ is minimal}\}=N\cap L_{-\nu},
$$ 
and we denote by $s(g)\le s'(g)$ the abscissas of the end points of $S_\nu(g)$, where  $\mu'=[\mu;\phi,\nu]$. 
If $N$ has a side $S$ of slope $-\nu$, then $S_\nu(g)=S$; otherwise, $S_\nu(g)$ is a vertex of $N$ and $s(g)=s'(g)$ (see Figure \ref{figComponent}).  

The next result facilitates the computation of the value $\mu'(g)$ from the Newton polygon $N_{\mu,\phi}(g)$.

\begin{lemma}\label{muprim}
With the above notation, the line $L_{-\nu}$ cuts the vertical axis at the point $(0,\mu'(g))$. Also, $s(g)=\ord_{\mu',\phi}(g)$. 
\end{lemma}

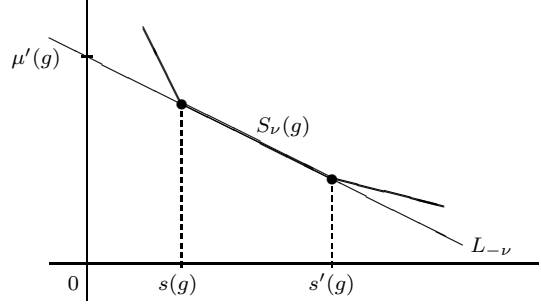
\begin{figure}
\caption{$\nu$-component of $N_{\mu,\phi}(g)$}\label{figComponent}
\begin{center}
\setlength{\unitlength}{5mm}
\begin{picture}(14,8)
\put(2.35,5.05){$\bullet$}\put(6.35,3.05){$\bullet$}
\put(-1,1){\line(1,0){13}}\put(0,0){\line(0,1){8}}
\put(-1,7){\line(2,-1){11}}
\put(2.5,5.25){\line(-1,2){1}}\put(2.5,5.3){\line(-1,2){1}}
\put(2.5,5.25){\line(2,-1){4}}\put(2.5,5.3){\line(2,-1){4}}
\put(6.5,3.25){\line(4,-1){3}}\put(6.5,3.28){\line(4,-1){3}}
\multiput(2.5,.9)(0,.25){17}{\vrule height2pt}
\multiput(6.5,.9)(0,.25){9}{\vrule height2pt}
\put(5.9,.3){\begin{footnotesize}$s'(g)$\end{footnotesize}}
\put(1.9,.3){\begin{footnotesize}$s(g)$\end{footnotesize}}
\put(10.2,1.3){\begin{footnotesize}$L_{-\nu}$\end{footnotesize}}
\put(-.5,.3){\begin{footnotesize}$0$\end{footnotesize}}
\put(4.5,4.5){\begin{footnotesize}$S_\nu(g)$\end{footnotesize}}
\put(-.15,6.5){\line(1,0){.3}}
\put(-2,6.3){\begin{footnotesize}$\mu'(g)$\end{footnotesize}}
\end{picture}
\end{center}
\end{figure}

\subsection{Inductive valuations}\label{subsecInductive}
A valuation $\mu\in\V$ is called \emph{inductive} if it is attained after a finite number of augmentation steps starting with $\mu_0$.
\begin{equation}\label{depth}
\mu_0\ \stackrel{\phi_1,\nu_1}\lra\  \mu_1\ \stackrel{\phi_2,\nu_2}\lra\ \cdots
\ \stackrel{\phi_{r-1},\nu_{r-1}}\lra\ \mu_{r-1} 
\ \stackrel{\phi_{r},\nu_{r}}\lra\ \mu_{r}=\mu.
\end{equation}
We denote by $\Vi\subset \V$ the subset of all inductive valuations. 

A chain of augmented valuations as in (\ref{depth}) is called a \emph{MacLane chain of length $r$} of $\mu$ if $\phi_{i+1}\not\sim_{\mu_i}\phi_i$ for all $1\le i<r$.  

We say that (\ref{depth}) is an \emph{optimal MacLane chain} of $\mu$ if  $\deg \phi_1<\cdots<\deg\phi_r$.  

An optimal MacLane chain is in particular a MacLane chain and every inductive valuation admits optimal MacLane chains \cite[Sec. 3.1]{Rideals}.

In every chain MacLane chain we have
$$
\deg \phi_i\mid \deg\phi_{i+1},\qquad
 \Gamma(\mu_i)\subset\Gamma(\mu_{i+1}),\quad 1\le i<r.
$$

\begin{proposition}\label{unicity}
Suppose the inductive valuation $\mu$ admits an optimal Mac\-Lane chain as in (\ref{depth}). Consider another optimal MacLane chain
$$
\mu_0\ \stackrel{\phi'_1,\nu'_1}\lra\  \mu'_1\ \stackrel{\phi'_2,\nu'_2}\lra\ \cdots
\ \stackrel{\phi'_{r'-1},\nu_{r'-1}}\lra\ \mu'_{r'-1} 
\ \stackrel{\phi'_{r'},\nu'_{r'}}\lra\ \mu'_{r'}=\mu'.
$$
Then, $\mu=\mu'$ if and only if $r=r'$ and:
$$\deg \phi_i=\deg\phi'_i, \quad \mu_i(\phi_i)=\mu_i(\phi'_i), \quad \nu_i=\nu'_i,\quad\mbox{ for all }\ 1\le i \le r.
$$
In this case, we also have $\mu_i=\mu'_i$ and $\phi_i\sim_{\mu_{i-1}}\phi'_i$ for all $1\le i \le r$. 
\end{proposition}

Therefore, in any optimal MacLane chain of $\mu$, the intermediate valuations $\mu_1,\dots,\mu_{r-1}$, the positive rational numbers $\nu_1,\dots,\nu_r$, and the degrees of the key polynomials $\deg\phi_1,\dots,\deg\phi_r$ are intrinsic data of $\mu$, whereas the key polynomials $\phi_1,\dots,\phi_r$ admit different choices. 

The \emph{MacLane depth} of an inductive valuation $\mu$ is the length of any optimal MacLane chain of $\mu$.  

A MacLane chain of $\mu$ determines an extension of $\mu$ to a valuation on $K_v(x)$. In fact, $\mu_0$ admits an obvious extension, and we may trivially extend to polynomials in $K_v[x]$ the definition of the successive augmentations. 

\begin{proposition}\label{KKv}
The restriction map
$\Vi(K_v)\to \Vi(K)$ is bijective. The inverse map $\Vi(K)\to \Vi(K_v)$ sends an inductive valuation $\mu$ on $K(x)$ to the valuation on $K_v(x)$ determined by a MacLane chain of $\mu$.
\end{proposition}

\subsection{Data and operators attached to a MacLane chain}\label{subsecNum}
Consider an inductive valuation $\mu$ equipped with a Maclane chain of length $r$ as in (\ref{depth}). We may attach to this chain several data and operators.  

Let us denote 
$$\Gamma_i=\Gamma(\mu_i)=e(\mu_i)^{-1}\Z, \quad \Delta_i=\Delta(\mu_i),\qquad 0\le i\le r.
$$
$$
\F_0:=\op{Im}(\F\to\Delta_0); \quad \F_i:=\op{Im}(\Delta_{i-1}\to\Delta_i),\quad 1\le i\le r.
$$
By Proposition \ref{sameideal}, $\F_i$ is a field canonically isomorphic to the residue class field $\F_{\phi_i}$ of the extension of $K_v$ determined by $\phi_i$; in particular, $\F_i$ is a finite extension of $\F$. We abuse of language and we identify $\F$ with $\F_0$ and each field $\F_i\subset\Delta_i$ with its image under the canonical map $ \Delta_i\to\Delta_j$ for $j\ge i$. In other words, we consider as inclusions the canonical embeddings
$$
\F=\F_0\subset \F_1\subset\cdots\subset \F_r.                                                                                                                                                                                                                                                                                                                                                                                                                                                                                                                                                                                                                                                                                                                                                                                         $$

Let us normalize the valuations $\mu_0,\dots,\mu_r$ by defining $v_i:=e(\mu_i)\mu_i$ for all $0\le i\le r$, so that $v_0,\dots,v_r$ have group of values equal to $\Z$.  

Take $e_0=m_0=1$ and $\nu_0=\la_0=h_0=w_0=V_0=0$.
For all $1\le i\le r$, we consider the following numerical data:
$$
\as{1.2}
\begin{array}{lll}
m_i:=\deg\phi_i,& e_i:=e(\mu_i)/e(\mu_{i-1}), &
f_{i-1}:=[\F_{i}\colon \F_{i-1}], \\
h_i:=e(\mu_i)\nu_i,&
\la_i:=e(\mu_{i-1})\nu_i=h_i/e_i,&\\
w_i:=\mu_{i-1}(\phi_i),& V_i:=e(\mu_{i-1})w_i=v_{i-1}(\phi_i),&\\ 
\end{array}
$$
It is easy to show that $\gcd(h_i,e_i)=1$. All these data may be expressed in terms of the positive integers 
\begin{equation}\label{MLinvariants}
e_0,\dots,e_r, \ f_0,\dots,f_{r-1},\  h_1,\dots,h_r.
\end{equation}
For instance, for all $1\le i\le r$ we have:
$$
\as{1.2}
\begin{array}{l}
e(\phi_i)=e(\mu_{i-1})=e_0\cdots e_{i-1},\\
f(\phi_i)=\left[\F_i\colon \F_0\right]=f_0\cdots f_{i-1},\\
\nu_i=h_i/e_1\cdots e_i,\\
m_i=e_{i-1}f_{i-1}m_{i-1}=(e_0\cdots e_{i-1})(f_0\cdots f_{i-1}),\\
w_i=e_{i-1}f_{i-1}(w_{i-1}+\nu_{i-1})=m_i\sum_{1\le j< i}\nu_j/m_j,\\
\end{array}
$$

Here, $e(\phi_i)$ and $f(\phi_i)$ are the ramification index and residual degree of the extension $K_{\phi_i}/K_v$, respectively. The recurrence satisfied by $m_i$, $w_i$ allows us to consider new data 
$$m_{r+1}:=e_rf_rm_r, \ w_{r+1}:=e_rf_r(w_r+\nu_r),\ V_{r+1}:=e(\mu_r)w_{r+1}=e_rf_r(e_rV_r+h_r).$$
 
If the MacLane chain is optimal, all these rational numbers are intrinsic data of $\mu$ by Proposition \ref{unicity}. In this case, we refer to them as $e_i(\mu)$, $f_i(\mu)$,
$h_i(\mu)$, $\la_i(\mu)$, $\nu_i(\mu)$, $m_i(\mu)$, $w_i(\mu)$, $V_i(\mu)$, and the positive integers in (\ref{MLinvariants}) are  called the \emph{basic MacLane invariants} of $\mu$.

We consider as well some rational functions in $K(x)$ defined in a recursive way. 
For every $0\le i\le r$, consider integers $\ell_i,\ell'_i$ uniquely determined by
\begin{equation}\label{Bezout}
\ell_i h_i+\ell'_i e_i=1,\qquad 0\le \ell_i<e_i.
\end{equation}
Take $\pi_0=\pi_1=\pi$, $\Phi_0=\phi_0=\gamma_0=x$, and define
\begin{equation}\label{ratfracs}
\Phi_i=\phi_i\,(\pi_{i})^{-V_i},\quad
\ga_i=(\Phi_i)^{e_i}(\pi_i)^{-h_i},\quad
\pi_{i+1}=(\Phi_{i})^{\ell_{i}}(\pi_{i})^{\ell'_{i}},\quad 1\le i\le r.
\end{equation}

It is easy to check by induction that
$$\mu_{i}(\pi_{i})=1/e(\mu_{i-1}), \quad \mu_{i}(\Phi_i)=\nu_i,\quad\mu_{i}(\ga_i)=0.
$$

All polynomial factors dividing $\pi_i$, and those dividing $\Phi_i$ with a negative exponent lead to units in the graded algebra of $\mu$. Hence, it makes sense to define, for all $0\le i\le r$:
$$
x_i:=H_{\mu_i}(\Phi_i)\in \gg(\mu_i),\ \;p_{i}:=H_{\mu_{i}}(\pi_i)\in \gg(\mu_i)^*,\ \;
y_i:=H_{\mu_{i}}(\gamma_i)=x_i^{e_i}p_i^{-h_i}\in\Delta_i,$$
and for $0\le i<r$:
$$
\begin{array}{l}
z_{i}\in\F_{i+1}, \ \mbox{the image of $y_{i}$ under } \Delta_{i}\to \Delta_{i+1},\\
\psi_{i}\in\F_{i}[y], \ \mbox{minimal polynomial of $z_{i}$ over }\F_{i}.
\end{array}
$$
We have $z_{i}\ne0$ (and $\psi_{i}\ne y$) for $i>0$.
For $i=0$ we have $z_0=0$ (and $\psi_0=y$) if and only if  $\overline{\phi}_1=x$ in $\F[x]$.
Moreover, 
$$\F_{i+1}=\F_{i}[z_{i}]=\F_0[z_0,\dots,z_{i}],\quad \deg\psi_{i}=f_{i}.
$$

Consider \emph{Newton polygon operators} 
$$N_i:=N_{v_{i-1},\phi_i}\colon\  K[x]\lra 2^{\R^2}, \quad 1\le i\le r.
$$ 
Since we deal with normalized valuations, the vertices of $N_i(g)$ have integer coordinates for any $g\in K[x]$. Actually, the Newton polygon $N_i(g)$ is the image of $N_{\mu_{i-1},\phi_i}(g)$ under the affine transformation $(x,y)\mapsto (x,e(\mu_{i-1})y)$. Hence, the vertices of both polygons have the same abscissas and this affine map sends the $\nu_i$-component of $N_{\mu_{i-1},\phi_i}(g)$ to the $\la_i$-component of $N_i(g)$. In particular, Lemma \ref{muprim} shows that the line of slope $-\la_i$ containing the  $\la_i$-component of $N_i(g)$ cuts the vertical axis at the ordinate $e(\mu_{i-1})\mu_i(g)=v_i(g)/e_i$ (see Figure \ref{figAlpha}).    
 
Also, a MacLane chain supports \emph{residual polynomial operators}:
$$
R_i:=R_{v_{i-1},\phi_i,\la_i}\colon K[x]\lra \F_i[y],\quad \ 0\le i\le r.
$$
We have $R_i(0)=0$ for all $i$. For a non-zero $g\in K[x]$ we define $R_0(g)=\overline{g/\pi^{\mu_0(g)}}$, whereas $R_i(g)$ for $i>0$ is determined by the following result.

\begin{theorem}\label{hmu}
For $i>0$ and a non-zero $g\in K[x]$ let $(s_i(g),u_i(g))$ be the left end point of the $\la_i$-component of $N_i(g)$. There exists a unique polynomial $R_i(g)\in\F_i[y]$ such that
$\op{H}_{\mu_i}(g)=
x_i^{s_i(g)}p_i^{u_i(g)}R_i(g)(y_i)$.
\end{theorem}

The degree of $R_i(g)$ is $(s'_i(g)-s_i(g))/e_i$, where $s'_i(g)$ is the abscissa of the right end point of the $\la_i$-component of $N_i(g)$.

In section \ref{secTypes} we shall show how to compute the operator $R_i$ in practice.

\subsection{Structure of the graded algebra}\label{subsecDelta} 
The elements $x_r,p_r,y_r\in\ggm$ attached to a MacLane chain determine the structure of the graded algera of an inductive valuation.
 
\begin{theorem}\label{Delta}
The mapping $\F_r[y]\to\Delta$ determined by $y\mapsto y_r$ is an isomorphism of $\F_r$-algebras. The inverse mapping is given by $$g+\pset_0^+(\mu)\mapsto y^{\lfloor s_r(g)/e_r\rfloor}R_r(g)(y),$$for any $g \in K[x]$ with $\mu(g)=0$. 
\end{theorem}

\begin{theorem}\label{structure}
The graded algebra of $\mu$ is
$$\ggm=\F_r[y_r,p_r,p_r^{-1}][x_r]=\Delta[p_r,p_r^{-1}][x_r].$$
The elements $y_r,p_r$ are algebraically independent over $\F_r$, and $x_r^{e_r}=p_r^{h_r}y_r$.
\end{theorem}

From these results one may derive further properties of the residual polynomials. The most outstanding fact is that the element $R_r(g)(y_r)\in\Delta$ is, up to a power of $y_r$, a generator of the residual ideal $\rr(g)$.

\begin{corollary}\label{Rprops}
Take $0 \le i\le r$ and non-zero $g,h\in K[x]$. Then, 
\begin{enumerate}
\item If $g\sim_{\mu_i}h$, then $R_i(g)=R_i(h)$.
\item If $i<r$, then $R_{i+1}(\phi_{i+1})=1$ and $R_i(\phi_{i+1})=\psi_i$. 
\item $R_i(gh)=R_i(g)R_i(h)$.
\item $\rr(g)=y_r^{\lceil s_r(g)/e_r\rceil}R_r(g)(y_r)\,\Delta$, where we agree that $s_0(g)=0$.
\item If $\phi$ is a key polynomial for $\mu$, then $\rr(\phi)=R_r(\phi)(y_r)\,\Delta$ if $\phi\not\smu \phi_r$, and $\rr(\phi)=y_r\Delta$ otherwise.
\end{enumerate}
\end{corollary}

The above results yield a strong connection between maximal ideals of $\Delta$ and residual ideals of key polynomials.  

\begin{theorem}\label{Max}
The mapping $\rr\colon \kpm\lra \mx(\Delta)$
induces a bijection  
between $\kpm/\!\smu$ and $\mx(\Delta)$.
\end{theorem}

\begin{corollary}\label{nextlength}
Let $\phi$ be a key polynomial for $\mu$ such that $\phi_r\nmid_\mu\phi$ and denote $\psi=R_r(\phi)$. Then, $\ord_\psi(R_r(g))=\ord_{\mu,\phi}(g)$ for any non-zero $g\in K[x]$.
\end{corollary}

\subsection{Data comparison between optimal MacLane chains}
Suppose that the given MacLane chain (\ref{depth}) of the inductive valuation $\mu$ is optimal. By Proposition \ref{unicity}, any other optimal MacLane chain of $\mu$ is obtained by replacing the key polynomials $\phi_1,\dots,\phi_r$  with another family $\phi^*_1,\dots,\phi^*_r$ such that
$$
\phi^*_i=\phi_i+a_i,\quad \deg a_i<m_i,\quad \mu_i(a_i)\ge \mu_i(\phi_i).
$$
Take $\eta_0:=0\in\F$. For every $1\le i\le r$ consider the following element $\eta_i\in\F_i$:
\begin{equation}\label{etai}
\eta_i:=
\begin{cases}
0,&  \mbox{ if }\mu_i(a_i)>\mu_i(\phi_i) \quad\mbox{ (i.e. }\phi^*_i\sim_{\mu_i}\phi_i),\\
R_i(a_i)\in\F_i^*,&  \mbox{ if }\mu_i(a_i)=\mu_i(\phi_i) \quad\mbox{ (i.e. }\phi^*_i\not\sim_{\mu_i}\phi_i).
\end{cases}
\end{equation}
Since $\deg a_i<\deg\phi_i$, we have $\mu_i(a_i)=\mu_{i-1}(a_i)$ by the definition of the augmentation of valuations. If $e_i>1$, we have $\mu_i(\phi_i)=\mu_{i-1}(\phi_i)+\nu_i\not \in\Gamma_{i-1}$. Hence, in this case we cannot have $\mu_i(a_i)=\mu_i(\phi_i)$. In other words,
$$
 e_i>1\ \imp\ \phi^*_i\sim_{\mu_i}\phi_i\ \imp\ \eta_i=0.
$$
The next result shows the relationship of the data $x_i,p_i,y_i,z_i,\psi_i$ attached to the optimal MacLane chain (\ref{depth}) with the analogous data $x^*_i,p^*_i,y^*_i,z^*_i,\psi^*_i$ attached to the optimal MacLane chain determined by the choice of $\phi^*_1,\dots,\phi^*_r$ as key polynomials.

\begin{lemma}\label{eta}
With the above notation, for all $0\le i\le r$ we have
$$
p^*_i=p_i,\quad x^*_i=x_i+p_i^{h_i}\eta_i,\quad y^*_i=y_i+\eta_i,
$$ 
whereas for $0\le i<r$ we have $z^*_i=z_i+\eta_i,\quad \psi^*_i(y)=\psi_i(y-\eta_i)$.
\end{lemma}

\section{Okutsu equivalence of prime polynomials}\label{secOkutsu} 
In this section, we show how inductive valuations parameterize certain sets of prime polynomials. All results are extracted from \cite{Rideals}. 

We shall apply inductive valuations $\mu$ on $K(x)$ to polynomials in $K_v[x]$, without any mention of the natural extension of $\mu$ to $K_v(x)$ described in Proposition \ref{KKv}.   

Let $\P\subset\oo_v[x]$ be the set of all monic irreducible polynomials in $\oo_v[x]$. We say that an element in $\P$ is a \emph{prime polynomial} (with respect to $v$). 

Let $F\in\P$ and fix $\t\in\kb$ a root of $F$. Let $K_F=K_v(\t)$ be the finite extension of $K_v$ generated by $\t$, $\oo_F$ the ring of integers of $K_F$, $\m_F$ the maximal ideal and $\F_F$ the residue class field. 
We have $\deg F=e(F)f(F)$, where $e(F)$, $f(F)$ are the ramification index and residual degree of $K_F/K_v$, respectively.

Let $\mu_{\infty,F}$ be the pseudo-valuation on $K[x]$ obtained as the composition:
$$
\mu_{\infty,F}\colon K[x]\lra K_v(\t)\stackrel{v}\lra \Q\cup\{\infty\},
$$the first mapping being determined by $x\mapsto \t$. This pseudo-valuation does not depend on the choice of $\t$ as a root of $F$. 

Recall that a pseudo-valuation has the same properties as a valuation, except for the fact that the pre-image of $\infty$ is a prime ideal which is not necessarily zero.

We are interested in finding properties of prime polynomials leading to a certain comprehension of the structure of the set $\P$. An inductive valuation $\mu$ such that $\mu<\mu_{\infty,F}$ reveals many properties of $F$. 

\begin{theorem}\label{vgt}
 Let $F\in \P$ be a prime polynomial.
An inductive valuation $\mu$ satisfies $\mu\le\mu_{\infty,F}$ if and only if there exists $\phi\in\kpm$ such that $\phi\mmu F$. In this case, for a non-zero polynomial $g\in K[x]$, we have
$$
\mu(g)=\mu_{\infty,F}(g) \quad\mbox{if and only if} \quad \phi\nmid_{\mu}g. 
$$
\end{theorem}
 
\begin{theorem}\label{fundamental}
Let $F$ be a prime polynomial, $\mu$ an inductive valuation and $\phi$ a key polynomial for $\mu$. Then, $\phi\mmu F$ if and only if $\mu_{\infty,F}(\phi)>\mu(\phi)$. Moreover, if this condition holds, then:
\begin{enumerate}
\item Either $F=\phi$, or the Newton polygon $N_{\mu,\phi}(F)$ is one-sided of slope $-\nu$, where $\nu=\mu_{\infty,F}(\phi)-\mu(\phi)\in\Q_{>0}$.    
\item Let $\ell=\ell(N_{\mu,\phi}(F))$. Then, $F\smu\phi^\ell$ and $\deg F=\deg \phi^\ell$.  
\end{enumerate}
\end{theorem}

Theorem \ref{fundamental} is a generalization of Hensel's lemma.  The residual ideal $\rr_\mu(F)=\rr_\mu(\phi)^\ell$ is a power of the maximal ideal $\rr_\mu(\phi)$. Thus, if for a certain polynomial $g\in K[x]$ the residual ideal $\rr_\mu(g)$ factorizes as the product of two coprime proper ideals, we may conclude that $g$ factorizes in $K_v[x]$. This yields the fundamental result concerning factorization of polynomials over $K_v$.

\begin{theorem}\label{main}
Let $\mu$ be an inductive valuation equipped with a MacLane chain of length $r$ as in (\ref{depth}). Let $\phi\in \kpm$ such that $\phi\not\smu\phi_r$. Then, every monic polynomial $g\in\oo_v[x]$ factorizes into a product of monic polynomials in $\oo_v[x]$:
$$
g=g_0\,\phi^{\ord_\phi(g)}\prod\nolimits_{(\la,\psi)} g_{\la,\psi},$$
where $-\la$ runs on the slopes of $N_{r+1}^-(g):=N_{v_r,\phi}^-(g)$ and $\psi$ runs on the prime factors of $R_{r+1,\la}(g):=R_{v_r,\phi,\la}(g)$ in $\F_{r+1}[y]$, where $\F_{r+1}:=\F_r[y]/(R_r(\phi))$. Moreover,
$$\deg g_0=\deg g-\ell(N^-_{r+1}(g))\deg\phi,\quad
\deg g_{\la,\psi}=e_\la\ord_\psi(R_{r+1,\lambda}(g))\deg\psi\deg\phi,
$$ 
where $e_\la$ is the least positive denominator of $\la$.
Further, if $\ord_\psi(R_{r+1,\la}(g))=1$, then $g_{\la,\psi}$ is irreducible in $\oo_v[x]$. 
\end{theorem}

\begin{proof}
Let $g=G_1\cdots G_t$ be the prime factorization of $g$ in $\oo_v[x]$. The factor $g_0$ is the product of all prime factors $G_j$ such that $\phi\nmid_\mu G_j$. The factor $\phi^{\ord_\phi(g)}$ is the product of all $G_j=\phi$. The factor $g_{\la,\psi}$ is the product of all $G_j$ such that $\phi\mmu G_j$, $N_{r+1}^-(G_j)$ is one-sided of slope $-\la$ and  $R_{r+1,\la}(G_j)$ is a power of $\psi$.    
\end{proof}

The \emph{Okutsu bound} of a prime polynomial $F\in\P$ is defined as 
$$
\delta_0(F):= \deg(F) \mx\left\{v(g(\t))/\deg g\mid g\in \oo[x],\ g\mbox{ monic},\ \deg g<\deg F\right\}.
$$

We may attach to $F$ a valuation $\mu_F\colon K_v(x)^*\to \Q$, determined by the following action on polynomials:
$$\mu_F(g)=\mn_{0 \le s}\{v(a_s(\t))+s\delta_0(F)\},
$$
where $g=\sum_{0\le s}a_sF^s$ is the $F$-expansion of $g$.

\begin{definition}\label{strong}
We say that a key polynomial $\phi$ for an inductive valuation $\mu$ of depth $r$ is \emph{strong} if either $r=0$ or $\deg \phi>m_r(\mu)$.
We say that $\ll\in\mx(\Delta(\mu))$ is \emph{strong} if $\ll=\rr(\phi)$ for a strong $\phi\in\kpm$.
\end{definition}

\begin{theorem}\label{OkML}
The mapping $\mu_F$ is an inductive valuation on $K_v(x)$ and $F$ is a strong key polynomial for $\mu_F$. 
\end{theorem}

We denote by the same symbol $\mu_F$ the valuation on $K(x)$ obtained by restriction.
The \emph{Okutsu depth} of a prime polynomial $F$ (defined in \cite{okutsu, Ok}) coincides with the MacLane depth of the canonical valuation $\mu_F$.

Let $F$ be a prime polynomial of Okutsu depth $r$, and define $f_r:=\deg R_r(F)$ with respect to any optimal MacLane chain of $\mu_F$.
An \emph{Okutsu invariant} of $F$ is a rational number that depends only on $e_0,\dots,e_r,f_0,\dots,f_r,h_1,\dots,h_r$; that is, on the basic MacLane invariants of $\mu_F$ and the number $f_r$.

As examples of Okutsu invariants we may quote:
\begin{equation}\label{efdelta}
e(F)=e(\mu_F)=e_0\cdots e_r, \quad f(F)=f_0\cdots f_r,\quad \delta_0(F)=w_{r+1}.
\end{equation}
In section \ref{subsecOM} we exhibit some more Okutsu invariants of prime polynomials.

\begin{definition}\label{okequiv}
Let $F,G\in\P$ be two prime polynomials of the same degree, and let $\t\in\kb$ be a root of $F$. 
We say that $F$ and $G$ are \emph{Okutsu equivalent}, and we write $F\approx G$, if $v(G(\t))>\delta_0(F)$.
\end{definition}

We denote by $[F]\subset\P$ the set of all prime polynomials which are Okutsu equivalent to $F$. The idea behind this concept is that $F$ and $G$ are close enough to share the same Okutsu invariants, as the next result shows.

\begin{proposition}\label{criteria}
Let $F,G\in\P$ be two prime polynomials of the same degree. The following conditions are equivalent:
\begin{enumerate}
\item $F\approx G$.
\item $F\sim_{\mu_F}G$.
\item $\mu_F=\mu_G$ and $\rr(F)=\rr(G)$, where $\rr:=\rr_{\mu_F}=\rr_{\mu_G}$.
\end{enumerate}
\end{proposition}

The symmetry of condition (3) shows that $\approx$ is an equivalence relation on the set $\P$ of prime polynomials. These conditions determine a parameterization of the quotient set 
$\P/\!\approx$ by a discrete space.

The \emph{MacLane space} of the valued field $(K,v)$ is defined to be the set
$$\M=\left\{(\mu,\ll)\mid \mu\in\Vi, \ \ll\in\mx(\Delta(\mu)), \ \ll \mbox{ strong}\right\}.$$

We may define the following ``Okutsu map":
$$
\op{ok}\colon \M\lra \P/\!\!\approx,\qquad (\mu,\ll)\ \mapsto\ [\phi], 
$$where $\phi$ is any key polynomial for $\mu$ such that $\rr_\mu(\phi)=\ll$.

\begin{theorem}\label{MLspace}
The Okutsu map is bijective and the inverse map is determined by $F\mapsto (\mu_F,\rr_{\mu_F}(F))$.
\end{theorem}

A point $(\mu,\ll)\in\M$ is characterized by discrete invariants which may be considered as  a kind of DNA sequence encoding arithmetic properties which are common to all prime polynomials in the Okutsu class $[F]=\op{ok}(\mu,\ll)$.

\section{Types over $(K,v)$}\label{secTypes} 
We keep dealing with a fixed discrete valued field $(K,v)$ with valuation ring $\oo$. 
 
\subsection{Types}\label{subsecTypes}
A \emph{type} is a computational object which is able to represent a pair $(\mu,\ll)$, where $\mu$ is an inductive valuation on $K(x)$ and $\ll$ is a maximal ideal in $\Delta(\mu)$. More precisely, a type collects discrete data determining a MacLane chain of $\mu$ and the maximal ideal $\ll$.

Therefore, a type $\ty$ supports some data structured into levels:
$$
\ty=(\psi_0;(\phi_1,\lambda_1,\psi_1);\cdots;(\phi_r,\lambda_r,\psi_r)).
$$ 
The number $r$ of levels is called the \emph{order} of the type.

A type $\ty=(\psi_0)$ of order $0$ is determined by the choice of an arbitrary monic irreducible polynomial $\psi_0\in\F[y]$. It supports the following data at level $0$:

\begin{itemize}
\item The minimal valuation $\mu_0$ on $K(x)$ and its normalization $v_0=\mu_0$.
\item Numerical data: $e_0=m_0=1$, $\nu_0=\lambda_0=h_0=0$.
\item $\psi_0\in\F_0[y]$ a monic irreducible polynomial. 
\item $\F_1=\F_0[y]/(\psi_0)$ a finite extension of $\F$ of degree $f_0:=\deg\psi_0$.
\item $z_0\in\F_1$ the class of $y$. Hence, $\F_1=\F_0[z_0]$ and $\psi_0$ is the minimal polynomial of $z_0$ over $\F_0$.
\item The residual polynomial operator $R_0\colon K[x]\to \F_0[y]$, where $\F_0=\F$. It is defined as $R_0(g)=\overline{g(y)/\pi^{v_0(g)}}$ for any non-zero $g\in K[x]$. 
\end{itemize}
 
If $\ty_0=(\psi_0;(\phi_1,\lambda_1,\psi_1);\dots;(\phi_{r-1},\lambda_{r-1},\psi_{r-1}))$ is a type of order $r-1\ge0$, then a type $\ty=(\ty_0;(\phi_r,\lambda_r,\psi_r))$ of order $r$ may be obtained by adding the following data at the $r$-th level:

\begin{itemize}
\item A \emph{representative} $\phi_r$ of $\ty_0$. That is, a monic polynomial $\phi_r\in\oo[x]$ of degree $m_r:=e_{r-1}f_{r-1}m_{r-1}$ such that $R_{r-1}(\phi_r)=\psi_{r-1}$. 
Lemma \ref{compare} below shows that $\phi_r$ is a key polynomial for $\mu_{r-1}$.
\item The Newton polygon operator $N_r=N_{v_{r-1},\phi_r}$. 
\item A positive rational number $\lambda_r=h_r/e_r$, with $h_r$, $e_r$ positive coprime integers. We say that $\lambda_r$ is the \emph{slope} of $\ty$ at level $r$.
\item The non-normalized slope $\nu_r=\la_r/e(\mu_{r-1})=h_r/e_1\cdots e_r$.
\item The augmented valuation $\mu_r=[\mu_{r-1};\phi_r,\nu_r]$, together with its normalization $v_r=e(\mu_r)\mu_r=e_1\cdots e_r\mu_r$.
\item $\psi_r\in\F_r[y]$ a monic irreducible polynomial, $\psi_r\ne y$. 
\item $\F_{r+1}=\F_r[y]/(\psi_r)$ a finite extension of $\F_r$ of degree $f_r:=\deg\psi_r$.
\item $z_r\in\F_{r+1}$ the class of $y$. Hence, $\F_{r+1}=\F_r[z_r]$ and $\psi_r$ is the minimal polynomial of $z_r$ over $\F_r$.
\item A residual polynomial operator $R_r\colon K[x]\to \F_r[y]$ described as follows. 
\end{itemize}
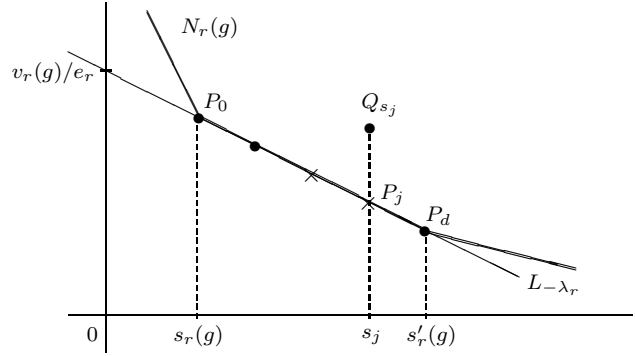
\begin{figure}
\caption{Computation of $R_r(g)$ for a non-zero polynomial $g\in K[x]$. The line $L_{-\la_r}$ has slope $-\la_r$.}\label{figAlpha}
\begin{center}
\setlength{\unitlength}{5mm}
\begin{picture}(14,8.6)
\put(2.3,5.05){$\bullet$}\put(3.8,4.3){$\bullet$}\put(5.2,3.55){$\times$}\put(6.7,2.8){$\times$}\put(8.3,2.05){$\bullet$}\put(6.85,4.8){$\bullet$}
\put(-1,0){\line(1,0){15}}\put(0,-1){\line(0,1){9.3}}
\put(-1,7){\line(2,-1){12}}
\put(2.5,5.25){\line(-1,2){1.4}}\put(2.5,5.29){\line(-1,2){1.4}}
\put(2.5,5.25){\line(2,-1){6}}\put(2.5,5.29){\line(2,-1){6}}
\put(8.5,2.2){\line(4,-1){4}}\put(8.5,2.24){\line(4,-1){4}}
\multiput(2.4,-.1)(0,.25){22}{\vrule height2pt}
\multiput(8.5,-.1)(0,.25){10}{\vrule height2pt}
\multiput(7,-.1)(0,.25){21}{\vrule height2pt}
\put(8.5,2.5){\begin{footnotesize}$P_{d}$\end{footnotesize}}
\put(2.6,5.5){\begin{footnotesize}$P_{0}$\end{footnotesize}}
\put(7.3,3.1){\begin{footnotesize}$P_j$\end{footnotesize}}
\put(6.8,5.5){\begin{footnotesize}$Q_{s_j}$\end{footnotesize}}
\put(6.8,-.6){\begin{footnotesize}$s_j$\end{footnotesize}}
\put(8,-.7){\begin{footnotesize}$s'_r(g)$\end{footnotesize}}
\put(1.8,-.7){\begin{footnotesize}$s_r(g)$\end{footnotesize}}
\put(11.2,.7){\begin{footnotesize}$L_{-\la_r}$\end{footnotesize}}
\put(2,7.5){\begin{footnotesize}$N_r(g)$\end{footnotesize}}
\put(-.5,-.7){\begin{footnotesize}$0$\end{footnotesize}}
\put(-.15,6.5){\line(1,0){.3}}
\put(-2.5,6.3){\begin{footnotesize}$v_r(g)/e_r$\end{footnotesize}}
\end{picture}
\end{center}
\end{figure}

The operator $R_r$ maps $0$ to $0$. For a non-zero $g\in K[x]$ with $\phi_r$-expansion $g=\sum_{0\le s}a_s\phi_r^s$, let us denote by $s_r(g)\le s'_r(g)$ the abscissas of the end points of the $\la_r$-component $S$ of $N_r(g)$ (cf. section \ref{subsecNewton}). Let $d=(s'_r(g)-s_r(g))/e_r$ be the \emph{degree} of $S$. There are $d+1$ points of integer coordinates $P_0,\dots,P_d$ lying on $S$, with abscissas $s_j:=s_r(g)+je_r$ for $0\le j\le d$ (see Figure \ref{figAlpha}). Denote by $Q_{s_j}=(s_j,v_{r-1}(a_{s_j}\phi_r^{s_j}))$ the point of abscissa $s_j$ in the cloud of points which is used to compute the Newton polygon  $N_r(g)$.
Consider the following residual coefficient:
\begin{equation}\label{rescoeffs}
c_j:=\begin{cases}
0,&\mbox{ if $Q_{s_j}$ lies above }N_r(g),\\
z_{r-1}^{t_{r-1}(a_{s_j})}R_{r-1}(a_{s_j})(z_{r-1})\in\F_r^*,&\mbox{ if $Q_{s_j}$ lies on }N_r(g),             
\end{cases}
\end{equation}
where for any $a\in K[x]$ we define $t_0(a)=0$ and $t_k(a)=(s_k(a)-\ell_kv_k(a))/e_k$ if $k>0$. Then, we define 
$$
R_r(g)(y):=R_{v_{r-1},\phi_r,\la_r}(g)=c_0+c_1y+\cdots+c_dy^d\in \F_r[y],
$$
Since $c_0c_d\ne0$, the polynomial $R_r(g)$ has degree $d$ and it is never divisible by $y$.

\begin{lemma}\label{compare}
Let $\ty$ be a type of order $r$ and denote $\mu:=\mu_r$, $\Delta:=\Delta(\mu)$.   
\begin{enumerate}
\item $\mu$ is an inductive valuation and the chain of augmentations
$$
\mu_0\ \stackrel{\phi_1,\nu_1}\lra\  \mu_1\ \stackrel{\phi_2,\nu_2}\lra\ \cdots
\ \stackrel{\phi_{r-1},\nu_{r-1}}\lra\ \mu_{r-1} 
\ \stackrel{\phi_{r},\nu_{r}}\lra\ \mu_{r}=\mu
$$
is a MacLane chain of $\mu$.
\item For $1\le i \le r$ denote by $\F_{i,\mu}$, $z_{i-1,\mu}$, $\psi_{i-1,\mu}$, $R_{i,\mu}$ the data and operator attached to this MacLane chain of $\mu$ in section \ref{subsecNum}. The rule $\iota_i(z_{i-1})=z_{i-1,\mu}$ determines a commutative diagram with vertical isomorphisms:
$$
\begin{array}{ccccccc}
\F=\F_{0}\ &\subset& \F_{1}\ &\subset&\cdots &\subset &\F_{r}\ \\  
\quad\ \,\|\ &&\ \downarrow\iota_1&&\cdots&&\ \downarrow\iota_r\\
\F=\F_{0,\mu}&\subset& \F_{1,\mu}&\subset&\cdots &\subset &\F_{r,\mu}
\end{array}
$$
If we denote still by $\iota_{i}$ the isomorphism between $\F_{i}[y]$ and $\F_{i,\mu}[y]$ induced by $\iota_{i}$, we have  $R_{i,\mu}=\iota_{i}\circ R_{i}$ for all $i$. Thus, up to considering these isomorphisms $\iota_i$ as identities,  we may identify all data and operators supported by $\ty$ with the analogous data and operators attached to $\mu$: 
$$\F_{i}=\F_{i,\mu},\quad z_{i-1}=z_{i-1,\mu},\quad \psi_{i-1}=\psi_{i-1,\mu},\quad R_{i}=R_{i,\mu}.$$ 
\item A polynomial $\phi\in K[x]$ is a representative of $\ty$ if and only if $\phi$ is a key polynomial for $\mu$ and $\rr(\phi)=\psi_r(y_r)\Delta$.\footnote{In this equality we  use the convention of item (2). The polynomial $\psi_r\in\F_{r}[y]$ is considered as a polynomial with coefficients in $\F_{r,\mu}\subset\Delta$ via the isomorphism $\iota_r\colon\F_{r}\to\F_{r,\mu}$.}   
\end{enumerate}
\end{lemma}

\begin{proof}
Let us prove all statements by induction on $r$. Suppose first that $\ty=(\psi_0)$ is a type of order $0$. In this case, $\mu=\mu_0$ and items (1) and (2) are trivial. Note that $R_{0}=R_{0,\mu}$ by the definition of both operators. A representative of $\ty$ is a monic polynomial $\phi\in \oo[x]$ of degree $m_1=f_0=\deg\psi_0$ such that $\overline{\phi}=R_0(\phi)=\psi_0$. On the other hand, a key polynomial for $\mu_0$ is a monic polynomial $\phi\in \oo[x]$ such that $\overline{\phi}$ is irreducible in $\F[x]$. Also, Corollary \ref{Rprops},(4) shows that $\rr(\phi)=R_0(\phi)(y_0)\Delta$. Since $R_0(\phi)$ and $\psi_0$ are monic polynomials, the equality $R_0(\phi)(y_0)\Delta=\psi_0(y_0)\Delta$ is equivalent to $R_0(\phi)=\psi_0$, by Theorem \ref{Delta}. This proves item (3). 

We assume from now on that $r>0$ and all statements of the lemma are true for types of order $r-1$. In particular, $\mu_{r-1}$ is an inductive valuation and
$$
\mu_0\ \stackrel{\phi_1,\nu_1}\lra\  \mu_1\ \stackrel{\phi_2,\nu_2}\lra\ \cdots\
\lra\ \mu_{r-2}\ \stackrel{\phi_{r-1},\nu_{r-1}}\lra\ \mu_{r-1}
$$
is a MacLane chain of $\mu_{r-1}$. For all $1\le i<r$ we have isomorphisms:
$$
\iota_i\colon \F_{i}\lra \F_{i,\mu},\quad z_{i-1}\mapsto z_{i-1,\mu}
$$
such that $\iota_i$ restricted to $\F_{i-1}$ coincides with $\iota_{i-1}$. Since $\psi_{i-1}$, $\psi_{i-1,\mu}$ are the minimal polynomials of $z_{i-1}$, $z_{i-1,\mu}$ over $\F_{i-1}$, $\F_{i-1,\mu}$, respectively, we have $\psi_{i-1,\mu}=\iota_{i-1}(\psi_{i-1})$. 
Also, $\phi_r$ is a key polynomial for $\mu_{r-1}$ such that 
$$
\rr_{\mu_{r-1}}(\phi_r)=\iota_{r-1}(\psi_{r-1})(y_{r-1})\Delta_{r-1}.
$$

In order to prove item (1) we need only to show that
$\phi_{r}\not\sim_{\mu_{r-1}}\phi_{r-1}$ if $r>1$. In fact, if $r>1$, then $\psi_{r-1}\ne y$; by Theorem \ref{Delta}, $\rr_{\mu_{r-1}}(\phi_r)\ne y_{r-1}\Delta_{r-1}$, and this implies $\phi_{r}\not\sim_{\mu_{r-1}}\phi_{r-1}$ by Corollary \ref{Rprops},(5).

Let us prove item (2). We have $\F_{r}=\F_{r-1}[z_{r-1}]$ and $\psi_{r-1}=R_{r-1}(\phi_r)$ is the minimal polynomial of $z_{r-1}$ over $\F_{r-1}$. Also, $\F_{r,\mu}=\F_{r-1,\mu}[z_{r-1,\mu}]$ and $\psi_{r-1,\mu}=R_{r-1,\mu}(\phi_r)$ (Corollary \ref{Rprops}) is the minimal polynomial of $z_{r-1,\mu}$ over $\F_{r-1,\mu}$. By the induction hypothesis, we have $R_{r-1,\mu}=\iota_{r-1}\circ R_{r-1}$, so that $\psi_{r-1,\mu}=\iota_{r-1}(\psi_{r-1})$, and this implies that $\iota_r$ is well-defined and is an isomorphism. 

By \cite[Def. 3.15 + Cor. 4.9]{Rideals}, for any non-zero $g\in K[x]$ we have 
$$
R_{r,\mu}(g)=c'_0+c'_1y+\cdots +c'_dy^d,
$$
where $d=(s'_r(g)-s_r(g))/e_r=\deg R_r(g)$ and the coefficients $c'_j\in\F_{r,\mu}$ satisfy:
$$
c'_j=\begin{cases}
0,&\mbox{if $Q_{s_j}$ lies above }N_r(g),\\
z_{r-1,\mu}^{\ell'_{r-1}\mathfrak{s}_j-\ell_{r-1}\mathfrak{u}_j}z_{r-1,\mu}^{\lfloor s_{r-1}(a_{s_j})/e_{r-1}\rfloor}R_{r-1,\mu}(a_{s_j})(z_{r-1,\mu}),&\mbox{if $Q_{s_j}$ lies on }N_r(g),             
\end{cases}
$$
where $a_{s_j}$, $Q_{s_j}$ are defined as in (\ref{rescoeffs}) and $\mathfrak{s}_j,\mathfrak{u}_j\in\Z$ are uniquely determined by:
\begin{equation}\label{leastpoint}
h_{r-1}\mathfrak{s}_j+e_{r-1}\mathfrak{u}_j=v_{r-1}(a_{s_j}),\quad 0\le \mathfrak{s}_j<e_{r-1}. 
\end{equation}

We want to prove that $R_{r,\mu}=\iota_{r}\circ R_{r}$, or equivalently $c'_j=\iota_r(c_j)$ for all $0 \le j\le d$, which is clearly equivalent to:
\begin{equation}\label{tocheck}
\ell'_{r-1}\mathfrak{s}_j-\ell_{r-1}\mathfrak{u}_j+\lfloor s_{r-1}(a_{s_j})/e_{r-1}\rfloor=
(s_{r-1}(a_{s_j})-\ell_{r-1}v_{r-1}(a_{s_j}))/e_{r-1}.
\end{equation}

Let $L$ be the line of slope $-\la_{r-1}$ containing the $\la_{r-1}$-component of $N_{r-1}(a_{s_j})$. As shown in Figure \ref{figAlpha}, this line cuts the vertical axis at the point $(0,v_{r-1}(a_{s_j})/e_{r-1})$. Hence, (\ref{leastpoint}) shows that $(\mathfrak{s}_j,\mathfrak{u}_j)$ is the point of least non-negative abscissa among all points on $L$ having integer coordinates. Since the point $(s_{r-1}(a_{s_j}), u_{r-1}(a_{s_j}))$ belongs to $L\cap(\Z_{\ge0}\times \Z)$, we have $\lfloor s_{r-1}(a_{s_j})/e_{r-1}\rfloor=(s_{r-1}(a_{s_j})-\mathfrak{s}_j)/e_{r-1}$. Then, the equality (\ref{tocheck}) is easily deduced from (\ref{leastpoint}) and the B\'ezout identity (\ref{Bezout}).

Let us prove item (3). After item (2), we may identify all data and operators supported by $\ty$ with the analogous data and operators attached to the MacLane chain of $\mu$. Suppose that $\phi$ is a representative of $\ty$, so that $\deg\phi=m_{r+1}$ and $\psi_r=R_r(\phi)$. 
By the definition of the operator $R_r$, we have 
$$
(s'_r(\phi)-s_r(\phi))m_r=e_r(\deg\psi_r) m_r=e_rf_rm_r=m_{r+1}=\deg\phi.
$$
Since $\deg \phi\ge s'_r(\phi)m_r$, we deduce that $s_r(\phi)=0$ and $s'_r(\phi)=e_r\deg\psi_r$.
Thus, $\phi$ is a key polynomial for $\mu$ because it satisfies condition (2) of \cite[Lem. 5.2]{Rideals}. By Corollary \ref{Rprops},(4), $\rr(\phi)=R_r(\phi)(y_r)\Delta=\psi_r(y_r)\Delta$. 

Conversely, suppose that $\phi\in \kpm$ satisfies $\rr(\phi)=\psi_r(y_r)\Delta$.  By Lemma \ref{irredKv}, $\phi$ is a monic polynomial with coefficients in $\oo$. Since $\psi_r\ne y$, Theorem \ref{Delta} and Corollary \ref{Rprops},(5) show that $\phi\not\smu \phi_r$ and $\rr(\phi)=R_r(\phi)(y_r)\Delta$. By \cite[Lem. 5.2]{Rideals}, $R_r(\phi)$ is monic irreducible and $\deg \phi=e_r\deg R_r(\phi) m_r$. By Theorem \ref{Delta}, the monic polynomials $\psi_r$ and $R_r(\phi)$ generate the same ideal in $\F_r[y]$; hence, $R_r(\phi)=\psi_r$ and $\deg\phi=m_{r+1}$. Thus, $\phi$ is a representative of $ \ty$.   
\end{proof}

Note that a type $\ty$ of order $r$ determines the numerical values $m_{r+1}:=e_rf_rm_r$, $V_{r+1}:=e_rf_r(e_rV_r+h_r)$ of any enlargement of $\ty$ to a type of order $r+1$.  

The data $\psi_r$, $\F_{r+1}$, $z_r$ at the $r$-th level of $\ty$ do not correspond to data attached to the MacLane chain of $\mu=\mu_r$. Through the isomorphism $\F_r[y]\simeq \Delta$ of Theorem \ref{Delta} the irreducible polynomial $\psi_r\in\F_r[y]$ determines a maximal ideal  $\ll=\psi_r(y_r)\Delta$ in $\Delta$. Hence, the type $\ty$ singles out a pair $(\mu_\ty,\ll_\ty)$, where $\mu_\ty=\mu$ is an inductive valuation and $\ll_\ty=\ll$ is a maximal ideal in $\Delta$. 

\begin{remark}
The definition of a type given in this paper has some slight differences with respect to the original definition in \cite{HN}, where $K$ was a global field.

(1) \ In \cite{HN} we used negative slopes $\lambda_i=-h_i/e_i$.

(2) \ The valuations $v_0,\dots,v_r$ were denoted $v_1,\dots,v_{r+1}$ in \cite{HN}. 

(3) \ Instead of the B\'ezout identities $\ell_ih_i+\ell'_ie_i=1$, in \cite{HN} we used the identities 
$\ell_ih_i-\ell'_ie_i=1$. This amounts to a change of sign of the data $\ell'_i$. 

(4) \ The residual operators $R_i$ have been normalized (by a slight change in the definition of the rational functions $\Phi_i$ from section \ref{subsecNum}) to satisfy $R_i(1)=1$. In this way, if $g\in K[x]$ has leading coefficient one in its $\phi_i$-expansion, then $R_i(g)$ is monic.   
\end{remark}

Let $\ty$ be a type of order $r$ over $(K,v)$. 
The \emph{truncation} of $\ty$ at level $j$, $\op{Trunc}_j(\ty)$, is the type of order $j$ obtained from $\ty$ by dropping all levels  higher than $j$.  

For any $g\in K[x]$ we define $\ord_\ty(g):=\ord_{\psi_{r}}R_{r}(g)$ in $\F_r[y]$. If $\ord_\ty(g)>0$, we say that $\ty$ \emph{divides} $g$, and we write $\ty\mid g$.  
By Corollary \ref{nextlength} and Lemma \ref{compare}, we have $\ord_\ty=\ord_{\mu_r,\phi}$ for any representative $\phi$ of $\ty$. In particular, $\ord_\ty(gh)=\ord_\ty(g)+\ord_\ty(h)$ for all $g,h\in K[x]$. 

The next result is a consequence of Proposition \ref{sameideal} and Theorem \ref{Delta}.

\begin{corollary}
Let $\ty$ be a type of order $r$, $\phi$ a representative of $\ty$, and $\alpha\in\kb$ a root of $\phi$. Then, we have an isomorphism
$$\F_{r+1}\iso \F_{\phi}, \quad z_0\mapsto \gamma_0(\alpha)+\m_\phi,
\dots,z_r\mapsto \gamma_r(\alpha)+\m_{\phi}
$$ where the rational functions $\gamma_0\dots,\gamma_r\in K(x)$ are those defined in (\ref{ratfracs}).
\end{corollary}

\subsection{Construction of types}\label{subsecConstruct}
Combined with Theorem \ref{Max}, Lemma \ref{compare},(3) shows that any type admits infinitely many representatives. In this section we describe a concrete procedure to construct a representative of a type. 

\begin{proposition}\label{construct}
Let $\ty$ be a type of order $r\ge 1$. Let $\varphi\in\F_r[y]$ be a non-zero polynomial of degree less than $f_r$ and let $b\ge V_{r+1}$ be an integer. Then, we may construct a polynomial $g\in \oo[x]$ such that $$\deg g<m_{r+1}, \quad v_r(g)=b,\quad  y^{\lfloor s_r(g)/e_r\rfloor}R_r(g)=\varphi.$$
\end{proposition}

\begin{proof}
Let $L$ be the line of slope $-\la_r$ cutting the vertical axis at the point $(0,b/e_r)$. Let $\mathfrak{s}$ be the least non-negative abscissa of a point of integer coordinates lying on $L$; this abscissa $\mathfrak{s}$ is uniquely determined by the conditions:
$$
\mathfrak{s}h_r\equiv b \md{e_r},\quad 0\le \mathfrak{s}<e_r.
$$
Let $k=\ord_y(\varphi)$ and write $\varphi=y^k\sum_{0\le j<f_r-k}\zeta_jy^j$, with $\zeta_j\in\F_r$ and $\zeta_0\ne0$. For each $0\le j<f_r-k$ such that $\zeta_j\ne0$ we denote 
$$
s_j=\mathfrak{s}+(j+k)e_r,\quad b_j=(b/e_r)-s_j(V_r+\la_r).
$$
Clearly, $s_j<(j+k+1)e_r\le e_rf_r$ and $b_j\ge (e_rf_r-s_j)(V_r+\la_r)> V_r+\la_r> V_r$, because $b/e_r\ge e_rf_r(V_r+\la_r)$ by hypothesis.

Also, for each such $j$ we consider an analogous abscissa $\mathfrak{s}_j$ determined by 
$$
\mathfrak{s}_jh_{r-1}\equiv b_j \md{e_{r-1}},\quad 0\le \mathfrak{s}_j<e_{r-1},
$$
and we let $\varphi_j\in\F_{r-1}[y]$ be the unique polynomial such that  
\begin{equation}\label{varphij}
\deg\varphi_j<f_{r-1},\quad \varphi_j(z_{r-1})=\zeta_j\, z_{r-1}^{(\ell_{r-1}b_j-\mathfrak{s}_j)/e_{r-1}}\in\F_r^*.
\end{equation}
For $r=1$ we have $\ell_0=0$, $\mathfrak{s}_j=0$ and $\varphi_j(z_{r-1})=\zeta_j$.

Consider $g=\phi_r^{s_0}\left(\sum_{0\le j<f_r-k}a_{s_j}\phi_r^{je_r}\right)$, where $a_{s_j}=0$ if $\zeta_j=0$, whereas for $\zeta_j\ne0$ we take $a_{s_j}\in\oo[x]$ satisfying
\begin{equation}\label{rec}
\deg a_{s_j}<m_r,\quad v_{r-1}(a_{s_j})=b_j,\quad y^{\lfloor s_{r-1}(a_{s_j})/e_{r-1}\rfloor}R_{r-1}(a_{s_j})=\varphi_j.
\end{equation}

Clearly, $\deg g<e_rf_rm_r$. Since $v_{r-1}(a_{s_j})=b_j$, the point $(s_j,v_{r-1}(a_{s_j}\phi_r^{s_j}))$ lies on $L$, and this guarantees that $v_r(g)=b$ by Lemma \ref{muprim}. By construction, $s_r(g)=s_0=\mathfrak{s}+ke_r$, so that 
$\lfloor s_r(g)/e_r\rfloor=k$. Thus, the condition $y^{\lfloor s_r(g)/e_r\rfloor}R_r(g)=\varphi$ is equivalent to $R_r(g)=\sum_{0\le j<f_r-k}\zeta_jy^j$; 
by the definition (\ref{rescoeffs}) of the coefficients of the residual polynomial, this amounts to $$z_{r-1}^{\left(s_{r-1}(a_{s_j})-\ell_{r-1}b_j\right)/e_{r-1}}R_{r-1}(a_{s_j})(z_{r-1})=\zeta_j$$
for all $0\le j<f_r-k$ such that $\zeta_j\ne0$. This equality is a consequence of (\ref{varphij}) and (\ref{rec}), having in mind that $\lfloor s_{r-1}(a_{s_j})/e_{r-1}\rfloor=(s_{r-1}(a_{s_j})-\mathfrak{s}_j)/e_{r-1}$ if $r>1$, whereas for $r=1$ we have $s_0(a_{s_j})=0$.

Therefore, we may construct $g$ by a recurrent procedure leading to the solution of the same problem for types of lower order. Thus, it suffices to solve the problem for types of order one, which is quite easy. In fact, if $r=1$ and $\zeta_j\ne0$, we may take $a_{s_j}=\pi^{b_j}a'_{s_j}$, where $a'_{s_j}$ is an arbitrary lifting of $\varphi_j\in\F[y]$ to $\oo[x]$; since $v_0(a'_{s_j})=0$, we have $v_0(a_{s_j})=b_j\ge V_1=0$, so that $a_{s_j}$ belongs to $\oo[x]$ as well. 
\end{proof}

In order to construct a representative $\phi$ of $\ty$ we may apply the procedure of Proposition \ref{construct} to construct a polynomial $g\in\oo[x]$ such that $R_r(g)=\psi_r-y^{f_r}$, and take $\phi=\phi_r^{e_rf_r}+g$. This justifies the following statement.

\begin{theorem}\label{representative}
We may efficiently construct representatives of types. 
\end{theorem}

Since the level data $\la_i$, $\psi_i$ are arbitrarily chosen, Theorem \ref{representative} shows that we may construct types of prescribed order $r$ and prescribed numerical data $h_i,e_i,f_i$ for $1\le i\le r$. In other words, we may construct inductive valuations of prescribed depth and prescribed MacLane invariants. This facilitates the construction of local extensions with prescribed arithmetic properties (cf. sections \ref{subsecOM} and \ref{subsecPrescribed}).

\subsection{Equivalence of types}Let  $\ty$ be a type of order $r\ge0$. We saw in section \ref{subsecTypes} that $\ty$ determines an inductive valuation $\mu_\ty$ and a maximal ideal $\ll_\ty$ in $\Delta:=\Delta(\mu_\ty)$. 

We say that $\ty$ is \emph{optimal} if $m_1<\cdots<m_r$.
We say that $\ty$ is \emph{strongly optimal} if $m_1<\cdots<m_r<m_{r+1}$. We agree that a type of order zero is strongly optimal.

\begin{lemma}\label{optimal}
The type $\ty$ is optimal if and only if the MacLane chain of $\mu_\ty$ attached to $\ty$ is optimal. In this case, the order of $\ty$ coincides with the depth of $\mu_\ty$.

The type $\ty$ is strongly optimal if and only if $\ty$ is optimal and $\ll_\ty$ is a strong maximal ideal of $\Delta$. 
\end{lemma}

\begin{proof}
The first statement is an immediate consequence of the definitions. 

Let $\ty$ be an optimal type with representative $\phi$. By Lemma \ref{compare}, $\phi$ is a key polynomial for $\mu_\ty$ and $\rr(\phi)=\ll_\ty$. Both conditions, $\ty$ strongly optimal, and $\ll_\ty$ strong (Definition \ref{strong}), are equivalent to $\deg\phi>m_r(\mu_\ty)$.
\end{proof}

The aim of this section is to extend the correspondence $\ty\mapsto (\mu_\ty,\ll_\ty)$ to an identification of the MacLane space $\M$ of $(K,v)$ with a quotient set of strongly optimal types classified by a certain equivalence relation. 
 
Denote by $\tcal$ the set of all types over $(K,v)$ and let $\tst\subset\tcal$ be the subset of all strongly optimal types. By Lemma \ref{optimal}, we have a well-defined ``MacLane map" from $\tst$ to the MacLane space of $(K,v)$:
$$
\op{ml}\colon \tst\lra  \M,\qquad \ty\mapsto (\mu_\ty,\ll_\ty).
$$
This mapping is clearly onto. In fact, for any point $(\mu,\ll)$ in the MacLane space $\M$ we may consider an optimal MacLane chain of $\mu$:
$$
\mu_0\ \stackrel{\phi_1,\nu_1}\lra\  \mu_1\ \stackrel{\phi_2,\nu_2}\lra\ \cdots
\ \stackrel{\phi_{r-1},\nu_{r-1}}\lra\ \mu_{r-1} 
\ \stackrel{\phi_{r},\nu_{r}}\lra\ \mu_{r}=\mu.
$$
Then, with the natural identifications described in Lemma \ref{compare}, this MacLane chain determines almost all data of an optimal type of order $r$:
$$
\ty=(\psi_0;(\phi_1,\lambda_1,\psi_1);\dots;(\phi_r,\lambda_r,-)),
$$ 
such that $\mu_\ty=\mu$. Also, the MacLane chain induces the isomorphism $\F_r[y]\simeq\Delta$ of Theorem \ref{Delta}, so that $\ll=\psi_r(y_r)\Delta$ for some (unique) monic irreducible polynomial $\psi_r\in \F_r[y]$. Hence, the optimal type $\ty=(\psi_0;(\phi_1,\lambda_1,\psi_1);\dots;(\phi_r,\lambda_r,\psi_r))$ satisfies $\mu_t=\mu$ and $\ll_\ty=\ll$. Since $\ll$ is a strong maximal ideal, the type $\ty$ is strongly optimal by Lemma \ref{optimal}.

Our next aim is to describe the fibers of the MacLane map. To this end we consider an equivalence relation on the set $\tst$ of strongly optimal types.

\begin{definition}\label{equivtypes}
Consider two strongly optimal types of the same order $r$:
$$\ty=(\psi_0;(\phi_1,\lambda_1,\psi_1);\dots;(\phi_r,\lambda_r,\psi_r)),\quad
\ty^*=(\psi^*_0;(\phi^*_1,\lambda^*_1,\psi^*_1);\dots;(\phi^*_r,\lambda^*_r,\psi^*_r)).$$
We say that $\ty$ and $\ty^*$ are equivalent if they satisfy the following conditions:
\begin{enumerate}
\item[(i)\ ] $\phi^*_i=\phi_i+a_i$, $\deg a_i<m_i$, $\mu_i(a_i)\ge\mu_i(\phi_i)$, for all $1\le i\le r$.
\item[(ii)\,] $\lambda^*_i=\lambda_i$ for all $1\le i\le r$.
\item[(iii)] $\psi^*_i(y)=\psi_i(y-\eta_i)$ with $\eta_i$ defined as in (\ref{etai}), for all $0\le i\le r$.
\end{enumerate}
We write $\ty\equiv\ty^*$ in this case. We denote by $\T=\tst/\equiv$ the quotient set and we write $[\ty]\subset \tst$ for the class of all types equivalent to $\ty$. 
\end{definition}

\begin{proposition}\label{sameml}
Two strongly optimal types $\ty,\ty^*$ are equivalent if and only if $\op{ml}(\ty)=\op{ml}(\ty^*)$.
\end{proposition}

\begin{proof}
If $\ty\equiv\ty^*$, then $\mu_\ty=\mu_{\ty^*}$ by Proposition \ref{unicity}. Also,
$$
\ll_{\ty^*}=\psi^*_r(y^*_r)\Delta=\psi_r(y^*_r-\eta_r)\Delta=\psi_r(y_r)\Delta=\ll_\ty,
$$ 
by Lemma \ref{eta}. Hence, $\op{ml}(\ty)=\op{ml}(\ty^*)$. 

Conversely, assume that $\op{ml}(\ty)=\op{ml}(\ty^*)$. From $\mu_\ty=\mu_{\ty^*}$ we deduce by  Proposition \ref{unicity} that conditions (i), (ii) from Definition \ref{equivtypes} hold, and condition (iii) holds for $i<r$. By Lemma \ref{eta}, we have moreover $y^*_r=y_r+\eta_r$. Hence, 
$$
\psi^*_r(y_r+\eta_r)\Delta=\psi^*_r(y^*_r)\Delta=\ll_{\ty^*}=\ll_\ty=\psi_r(y_r)\Delta.
$$  
Since these polynomials are monic, Theorem \ref{Delta} shows that $\psi^*_r(y+\eta_r)=\psi_r(y)$. 
\end{proof}

In combination with Theorem \ref{MLspace}, we get the following result.

\begin{theorem}\label{mlok}
The MacLane and Okutsu maps induce a canonical bijection between the set of equivalence classes of strongly optimal types and the set of Okutsu equivalence classes of prime polynomials: 
$$
\T\ \stackrel{\op{ml}}\lra\ \M\ \stackrel{\op{ok}}\lra\ (\P/\approx).
$$ 
\end{theorem}

\begin{corollary}\label{RepOk}
If $\phi$ is a representative of $\ty\in\tst$, then $(\op{ok}\circ\op{ml})([\ty])=[\phi]$ and $[\phi]\cap\oo[x]$ coincides with the set $\op{Rep}(\ty)$ of all representatives of $\ty$. 
\end{corollary}

\begin{proof}
An immediate consequence of Proposition \ref{criteria} and Lemma \ref{compare},(3).  
\end{proof}

\subsection{Tree structure on the set of types}
Let us introduce a tree structure on the set $\tcal$ of types. Given two types $\ty,\ty'\in\tcal$, there is an oriented edge $\ty'\to\ty$ if and only if $\ty'=\op{Trunc}_{r-1}(\ty)$, where $r$ is the order of $\ty$. Thus, we have a unique path of length equal to the order of $\ty$:
\begin{equation}\label{path}
\op{Trunc}_{0}(\ty)\lra \op{Trunc}_{1}(\ty)\lra \cdots\lra \op{Trunc}_{r-1}(\ty)\lra \ty.
\end{equation}
The root nodes are the types of order zero. Thus, the connected components of $\tcal$ are the subtrees $\tcal_\varphi$ of all types $\ty$ with $\op{Trunc}_{0}(\ty)=(\varphi)$, for $\varphi$ running on the set $\P(\F)$ of all monic irreducible poynomials in $\F[y]$.

The branches of a type $\ty$ of order $r$ are parametrized by triples $(\phi,\la,\psi)$, where $\phi$ is a representative of $\ty$, $\la$ is a positive rational number and $\psi\in\F_{r+1}[y]$ is a monic irreducible polynomial such that $\psi\ne y$. Such a triple determines an edge $\ty\to \ty^*$, where $\ty^*=(\ty;(\phi,\la,\psi))$ is the type obtained by enlarging $\ty$ with data $(\phi,\la,\psi)$ at the $(r+1)$-th level. 

Suppose $\ty=(\psi_0;(\phi_1,\la_1,\psi_1);\dots;(\phi_r,\la_r,\psi_r))$. In practice, when we represent a path like (\ref{path}) we omit the labels of the vertices which are not root nodes and we label the edges with the level data.

\begin{equation}\label{figPath}
\setlength{\unitlength}{5mm}
\begin{picture}(14,.5)
\put(-.2,0){$\bullet$}\put(3.8,0){$\bullet\qquad \ \cdots\cdots $}\put(9.8,0){$\bullet$}\put(13.8,0){$\bullet$}
\put(-1,.1){\begin{footnotesize}$\psi_0$\end{footnotesize}}
\put(0,.2){\line(1,0){4}}\put(10,.2){\line(1,0){4}}
\put(.5,0.6){\begin{footnotesize}$(\phi_1,\la_1,\psi_1)$\end{footnotesize}}
\put(10.5,0.6){\begin{footnotesize}$(\phi_{r},\la_{r},\psi_{r})$\end{footnotesize}}
\end{picture}
\end{equation}

Also, since the sense of the edges is self-evident, we draw them as lines instead of vectors. 
We recover the real path (\ref{path}) from its practical representation (\ref{figPath}) by attaching to each vertex of the path the type obtained by gathering all level data from the previous edges.

All truncates of a strongly optimal type $\ty$ are strongly optimal, hence the subset $\tst\subset \tcal$ is a full subtree of $\tcal$.  
Also, if $\ty\equiv\ty^*$ are strongly optimal, then $\op{Trunc}_i(\ty)\equiv\op{Trunc}_i(\ty^*)$ for all $0\le i\le r$. Therefore, the tree structure on $\tst$ induces a natural tree structure on the quotient set $\T=\tst/\equiv$. 

Since the equivalence relation $\equiv$ only identifies vertices of the same order, a path of length $r$ in $\tst$ determines a path of length $r$ in $\T$. 

For types of order zero, $\ty\equiv\ty^*$ holds only for $\ty=\ty^*$; thus, the root nodes of $\T$ are in 1-1 correspondence with the set $\P(\F)$ too.

The branches of $[\ty]\in\T$ are determined by triples $(\phi,\la,\psi)$ as above such that $e_\la\deg\psi>1$, where $e_\la$ is the least positive denominator of $\la$. Two such triples $(\phi,\la,\psi)$, $(\phi^*,\la^*,\psi^*)$ yield the same branch if and only if $(\ty;(\phi,\la,\psi))\equiv (\ty;(\phi^*,\la^*,\psi^*))$; by Definition \ref{equivtypes} this is equivalent to
$$
\la^*=\la, \quad \mu_\ty(\phi-\phi^*)\ge\mu_\ty(\phi)+\la/(e_1\cdots e_{r-1}),\quad \psi^*(y)=\psi(y-\eta),
$$
with $\eta=\eta_{r+1}$ defined as in (\ref{etai}) with respect to $\phi_{r+1}=\phi$ and $\phi^*_{r+1}=\phi^*$.

Of course, through the bijective mappings $\op{ml}$ and $\op{ok}$ we obtain a tree structure on the sets $\M$ and $\P/\approx$ as well.

\section{OM representations of square-free polynomials}\label{secOM}

\subsection{OM representations of prime polynomials}\label{subsecOM}
Consider a prime polynomial $F\in\P$ and let $(\mu,\ll)\in \M$ be the point in the MacLane space corresponding to the Okutsu equivalence class of $F$; that is,
$\op{ok}(\mu,\ll)=[F]$.

For any polynomial $\phi\in[F]\cap \oo[x]$ the pair $[(\mu,\ll),\phi]$
is called an \emph{OM representation} of $F$. If $\phi=F$ we say that the OM representation is \emph{exact}.

By Theorem \ref{mlok} and Corollary \ref{RepOk}, an OM representation may be handled in a computer as a pair
$$
[(\mu,\ll),\phi]\quad \leftrightarrow\quad [\ty,\phi],
$$
where $\ty$ is a strongly optimal type of order $r$ such that $\op{ml}([\ty])=(\mu,\ll)$, and $\phi$ is a representative of $\ty$. Note that
$$
\mu=\mu_\ty=\mu_F=\mu_\phi,\qquad \ll=\ll_\ty=\rr(F)=\rr(\phi).
$$

The polynomial $\phi$ is a ``sufficiently good" approximation to $F$ for many purposes. In a computational context, we propose to manipulate prime polynomials via OM representations $[\ty,\phi]$ instead of dealing barely with approximations with a given precision. The discrete data contained in the type $\ty$ is a kind of DNA sequence common to all individuals in the Okutsu class $[F]$, and many properties of $F$ and the extension $K_F/K_v$ are described by this genetic data. 

This approach has many advantages. 
The genetic data of $F$ provide arithmetic information on $F$ and $K_F$, which in the classical approach has to be derived from extra routines that may have a heavy cost. Further, the genetic information of $F$ is helpful in the construction of approximations with a prescribed quality  and, more generally, it leads to a new design of fast routines carrying out basic arithmetic tasks in number fields and function fields. Finally, the constructive procedure of section \ref{subsecConstruct} may be used to efficiently construct
prime polynomials with prescribed genetic data, or equivalently, with prescribed arithmetic properties. 

These algorithmic applications are discussed in section \ref{secComp}.
Let us now mention a few concrete facts that illustrate some of these advantages. Let $r$ be the Okutsu depth of $F$, $n$ the degree of $F$ and let us fix $\t\in\kb$ a root of $F$. 

\subsection*{Maximal tamely ramified subextensions}If the residue class field $\F$ is a perfect field and $F$ is a separable polynomial, then the extensions $K_\phi/K_v$ and $K_F/K_v$ have isomorphic maximal tamely ramified subextensions \cite{okutsu,Ok}. In particular, if $K_\phi/K_v$ is tamely ramified then $K_\phi$ and $K_F$ are isomorphic.

\subsection*{Okutsu bases}
The ring $\oo_F$ is a free $\oo_v$-module of rank $n$ and a basis is determined by the genetic information \cite{Ok}.

We may express any integer $0\le m<n$ in a unique way as:
$$
m=j_0+j_1m_1+\cdots+j_rm_r,\quad 0\le j_i<e_if_i.
$$
Consider the following integer $d_m$ and polynomial $g_m$ of degree $m$:
$$
d_m=\lfloor j_1(w_1+\nu_1)+\cdots+j_r(w_r+\nu_r)\rfloor,\quad g_m(x)=\phi_0(x)^{j_0}\phi_1(x)^{j_1}\cdots \phi_r(x)^{j_r}.
$$
Then, the following family is an $\oo_v$-basis of $\oo_F$:
$$
1,\,\pi^{-d_1}g_1(\t),\,\dots,\,\pi^{-d_{n-1}}g_{n-1}(\t).
$$

\subsection*{Okutsu invariants}
All Okutsu invariants of $F$ may be deduced from an OM representation of $F$ by closed formulas. For instance, let us exhibit some more Okutsu invariants, taken from \cite[Sec. 1]{Ndiff}, besides $e(F)$, $f(F)$ and $\delta_0(F)$ already mentioned in (\ref{efdelta}).  
\begin{equation}\label{Okinvariants}
\as{1.2}
\begin{array}{l}
\op{cap}(F):=\mx\left\{v(g(\t))\mid g\in\oo[x] \mbox{ monic, } \deg g<n \right\}=w_{r+1}-\sum_{j=1}^r\nu_j,\\
\exp(F):=\mn\left\{\delta\in\Z_{\ge0}\mid \m^\delta \oo_F\subset \oo_v[\t]\right\}=\lfloor\op{cap}(F)\rfloor,\\
\ind(F):=\operatorname{length}_{\oo_v}\left(\oo_F/\oo_v[\t]\right)=n\left(\op{cap}(F)-1+e(F)^{-1}\right)/2,\\
\mathfrak{f}(F):=\mn\left\{\delta\in\Z_{\ge0}\mid (\m_F)^\delta \subset \oo_v[\t]\right\}=2\ind(F)/f(F).
\end{array}
\end{equation}

These numbers are called the \emph{capacity}, \emph{exponent}, \emph{index} and \emph{conductor} of $F$, respectively. The notation $\operatorname{length}_{\oo_v}$ indicates length as an $\oo_v$-module. 

\subsection*{Quality of an approximation}
There are two typical measures of the distance between $\phi$ and $F$:
$$
\nu=\mu_0(F-\phi), \quad \nu'=v(\phi(\t))=\mu_{\infty,F}(F-\phi),
$$
called the \emph{precision} and the \emph{quality} of the approximation, respectively. The precision is the largest positive integer $\nu$ such that $F\equiv \phi\md{\m^\nu}$, whereas the \emph{quality} is a positive rational number

Usually, $F$ is a prime factor of some given polynomial $f\in\oo[x]$. We shall see in section \ref{subsecGenomic} that in this case  
\begin{equation}\label{quality}
\nu'=w_{r+1}+\nu_{r+1}=\delta_0(F)+\la_{r+1}/e(F),
\end{equation}
where $\lambda_{r+1}$ is a positive integer which may be read in $N_{r+1}^-(f):=N_{v_r,\phi}^-(f)$.   

The two measures are related by the following inequalities. The first one is obvious and the second one was derived in \cite[Lem. 4.5]{okutsu}.

\begin{lemma}\label{measures}
For any OM representation $[\ty,\phi]$ of $F$, we have
$$
\nu'\ge\nu\ge \nu'-\op{cap}(F)=\nu_1+\cdots+\nu_r+\nu_{r+1}.
$$
\end{lemma}

Let us exhibit some examples showing that both inequalities are sharp and illustrating that $\nu'$ is a better measure than $\nu$ of the distance between $F$ and $\phi$.\medskip

\noindent{\bf Examples. } 
If $\phi=F+\pi^m$, then $\nu=m=\nu'$ and the first inequality of Lemma \ref{measures} is sharp. If the Okutsu depth of $F$ is $r\ge 1$ and we take $\phi=F+\pi^m\phi_1^{(n/m_1)-1}$, then $\nu=m$, whereas the quality
$$
\nu'=m+((n/m_1)-1)v(\phi_1(\t))=\nu+((n/m_1)-1)\nu_1
$$
can be much larger than $\nu$ if $n/m_1$ and/or $\nu_1$ are large. 

For instance, the prime polynomial $F=x^2+\pi$ is a representative of the type $\ty=(y;(x,1/2,y+1))$; hence, it has invariants $m_1=1$, $e_1=2$, $f_1=h_1=1=w_2$ and $\op{cap}(F)=w_2-\nu_1=1/2$. For the approximation $\phi=x^2+\pi^mx+\pi$ we have $\nu'=m+\nu_1=\nu+(1/2)$, so that the second inequality of Lemma \ref{measures} is sharp. We deduce that $\nu_2=m-(1/2)$ and $\la_2=2m-1$.

\subsection{OM representation of a square-free polynomial}\label{subsecGenomic}
Let $f=F_1\cdots F_t$ be the prime factorization in $\oo_v[x]$ of a square-free monic polynomial $f\in\oo[x]$. For each  $1\le j\le t$, let $r_j$ be the Okutsu depth of $F_j$ and  
$\t_j\in\kb$ a root of $F_j$. 

For a prime polynomial $F\in\P$, we denote by $\ty_F$ any strongly optimal type whose equivalence class corresponds to the Okutsu class of $F$ under the mapping $\op{ok}\circ\op{ml}$ of Theorem \ref{mlok}. That is, $$[\ty_F]=(\op{ok}\circ\op{ml})^{-1}([F])\in\T.$$ 

\begin{definition}\label{genomictree}
We denote by $\T(F)\subset \T$ the unibranch tree determined by the path joining $[\ty_F]$ with its root node in $\T$.  
The \emph{genomic tree} of $f$ is the finite tree $\T(f):=\T(F_1)\cup\dots\cup\T(F_t)\subset \T$.
\end{definition}

An OM representation of $f$ is an object which gathers the information provided by a family of OM representations of the prime factors. The approximations to the prime factors contained in all these OM representations constitute an approximate factorization of $f$ in $\oo_v[x]$.  
Since we are only interested in approximate factorizations which are able to distinguish the different prime factors of $f$, we are led to consider the so-called \emph{OM factorizations} of $f$.

\begin{definition}\label{okequiv2}
Let $g,h\in\oo[x]$ be monic polynomials with prime factorizations $g=G_1\cdots G_s$, $h=H_1\cdots H_{s'}$ in $\oo_v[x]$. 
We say that $g$ and $h$ are \emph{Okutsu equivalent}, and we write $g\approx h$, if $s=s'$ and $G_j\approx H_j$ for all $1\le j\le s$, up to ordering.

An expression of the form, $g\approx P_1\cdots P_s$, with $P_1,\dots,P_s\in\P\cap \oo[x]$ is called an \emph{Okutsu factorization} of $g$. 
\end{definition}

Clearly, every $g\in\oo[x]$ admits a unique (up to $\approx$) Okutsu factorization. However, we need a stronger concept for our purposes. For instance, if all factors of $g$ are Okutsu equivalent to $P$, then $g\approx P^t$ is an Okutsu factorization of $g$ which is unable to distinguish the true prime factors of $g$. 

\begin{definition}\label{OMfactorization}
We say that $P_j\in[F_j]$ is a \emph{Montes approximation to $F_j$ as a factor of $f$} if
$\ v(P_j(\t_j))>v(P_j(\t_{k}))$ for all $k\ne j$.

An \emph{OM factorization} of $f$ is an Okutsu factorization $f\approx P_1\cdots P_t$ such that each approximate factor $P_j$ is a Montes approximation to $F_j$ as a factor of $f$.  
\end{definition}

Let $f\approx P_1\cdots P_t$ be an OM factorization of $f$. By Corollary \ref{RepOk}, $P_j$ is a representative of $\ty_{F_j}$ and $[\ty_{F_j},P_j]$ is an OM representation of $F_j$ for all $j$. 

In \cite[Sec. 3.1]{BNS} it is shown that the types $\ty_{F_j}$ may be extended to types  
$$
\ty_j:=\left(\ty_{F_j};(P_j,\la_{r_j+1,j},\psi_{r_j+1,j})\right)\ \mbox{ or }\ \ty_j:=\left(\ty_{F_j};(P_j,\infty,-)\right),
$$
according to $P_j\ne F_j$ or $P_j=F_j$, respectively. These types of order $r_j+1$ satisfy
$$
\ord_{\ty_j}(F_j)=1,\quad \ty_j\nmid F_k,\quad \mbox{for all }1\le k\ne j\le t.
$$
The quality of the approximations $P_j\approx F_j$ is given by the formula: $$v(P_j(\t_j))=\delta_0(F_j)+\la_{r_j+1,j}/e(F_j).
$$

If $P_j\nmid f$, the slope $\la_{r_j+1,j}$ is an integer which may be computed as the largest slope (in absolute value) of $N_{r_j+1}^-(f)=N_{v_{r_j},P_j}^-(f)$. This slope corresponds to a side whose end points have abscissas $0$ and $1$ (see Figure \ref{figLastN}). Hence, $R_{r_j+1}(f):=R_{v_{r_j},P_j,\la_{r_j+1}}(f)$ has degree one and $\psi_{r_j+1,j}$ is equal to $R_{r_j+1}(f)$  divided by its leading coefficient.

The types $\ty_j$ are optimal, but not strongly optimal because  $e_{r_j+1}=f_{r_j+1}=1$, so that  $m_{r_j+2}=m_{r_j+1}=\deg F_j$.    

\begin{definition}\label{OMrepf}
Let $T(f)\subset \tst$ be a faithful preimage of the genomic tree of $f$; that is, $T(f)$ maps to $\T(f)$ under the quotient map $\tst\to\T$, and the vertices of $T(f)$ are pairwise inequivalent. 

An \emph{OM representation} of $f$ is the tree obtained by enlarging $T(f)$ with the $t$ new vertices $\ty_j$ and edges $\ty_{F_j}\to \ty_j$ determined by some OM factorization of $f$.
\end{definition}

The leaves of an OM representation of $f$ are in 1-1 correspondence with the prime factors of $f$, whereas the root nodes are in 1-1 correspondence with the monic irreducible factors of $\overline{f}$ in $\F[y]$. Let us see some examples where $\overline{f}$ is supposed to be a power of an irreducible polynomial in $\F[y]$, so that the OM representation of $f$ is a connected tree.
 
Let $f=F_1F_2$ be a polynomial with two Okutsu equivalent prime factors. Then, $[\ty_{F_1}]=[\ty_{F_2}]$ and the genomic tree $\T(f)=\T(F_1)=\T(F_2)$ is a unibranch tree as in (\ref{figPath}). It contains the genetic information of all prime factors of $f$, but it does not make apparent how to distinguish these factors. 

An OM representation of $f$ gives a more precise view of the different prime factors of $f$ and their genetic information:

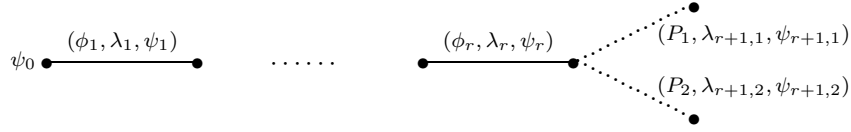
\begin{center}
\setlength{\unitlength}{5mm}
\begin{picture}(20,3)
\put(-.2,1.2){$\bullet$}\put(3.8,1.2){$\bullet\qquad \ \cdots\cdots $}\put(9.8,1.2){$\bullet$}\put(13.8,1.2){$\bullet$}
\put(-1,1.3){\begin{footnotesize}$\psi_0$\end{footnotesize}}
\put(0,1.4){\line(1,0){4}}\put(10,1.4){\line(1,0){4}}
\put(.5,1.8){\begin{footnotesize}$(\phi_1,\la_1,\psi_1)$\end{footnotesize}}
\put(10.5,1.8){\begin{footnotesize}$(\phi_{r},\la_{r},\psi_{r})$\end{footnotesize}}
\put(17,2.7){$\bullet$}\put(17,-.3){$\bullet$}
\put(16.2,2){\begin{footnotesize}$\left(P_1,\la_{r+1,1},\psi_{r+1,1}\right)$\end{footnotesize}}
\put(16.2,0.6){\begin{footnotesize}$\left(P_2,\la_{r+1,2},\psi_{r+1,2}\right)$\end{footnotesize}}
\multiput(14,1.4)(.2,.1){15}{.}
\multiput(14,1.4)(.2,-.1){15}{.}
\end{picture}
\end{center}\medskip

We represent the edges $\ty_{F_j}\to \ty_j$ with dotted lines to emphasize that the leaves $\ty_j$ are not strongly optimal types.

In general, the vertices $\ty_{F_i}$ are not necessarily leaves of the tree $T(f)$. It may happen that $\ty_{F_i}$ concides with a vertex in the path joining $\ty_{F_j}$ with its root node for some $j\ne i$. Thus, the leaves of an OM representation of $f$ may sprout from arbitrary vertices in $T(f)$. For instance, in the next example $f$ has four prime factors; the vertex $\ty_{F_1}$ has order $0$, $\ty_{F_2}=\ty_{F_3}$ have order $3$ and $\ty_{F_4}$ has order $5$.

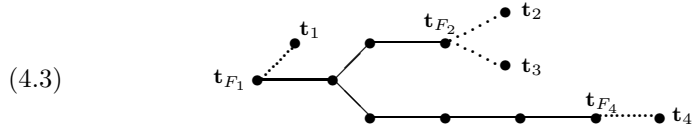
\begin{equation}\label{sampletree}
\setlength{\unitlength}{5mm}
\begin{picture}(12,2)
\put(-.2,0){$\bullet$}\put(.8,1){$\bullet$}\put(1.8,0){$\bullet$}
\put(0,0.2){\line(1,0){2}}\multiput(0,0.2)(.1,.1){10}{.}
\put(2,0.2){\line(1,1){1}}\put(2,0.2){\line(1,-1){1}}
\put(2.8,-1){$\bullet$}\put(2.8,1){$\bullet$}
\put(3,1.2){\line(1,0){2}}
\put(3,-.8){\line(1,0){2}}\put(5,-.8){\line(1,0){2}}\put(7,-.8){\line(1,0){2}}
\put(1.1,1.4){\begin{footnotesize}$\ty_1$\end{footnotesize}}
\put(-1.2,.1){\begin{footnotesize}$\ty_{F_1}$\end{footnotesize}}
\put(4.8,1){$\bullet$}
\put(4.8,-1){$\bullet$}\put(6.8,-1){$\bullet$}\put(8.8,-1){$\bullet$}
\multiput(5,1.2)(.2,.1){7}{.}\multiput(5,1.2)(.2,-.1){7}{.}
\put(6.4,1.8){$\bullet$}\put(7,1.9){\begin{footnotesize}$\ty_2$\end{footnotesize}}
\put(4.4,1.6){\begin{footnotesize}$\ty_{F_2}$\end{footnotesize}}
\put(6.4,.4){$\bullet$}\put(7,.4){\begin{footnotesize}$\ty_3$\end{footnotesize}}
\multiput(9,-.8)(.15,0){10}{.}
\put(10.5,-1){$\bullet$}\put(11,-.9){\begin{footnotesize}$\ty_4$\end{footnotesize}}
\put(8.7,-.5){\begin{footnotesize}$\ty_{F_4}$\end{footnotesize}}
\end{picture}
\end{equation}\medskip

We define the \emph{index of coincidence} $i([\ty],[\ty'])$ between two vertices $[\ty],[\ty']\in\T$, as follows. If they have different root nodes we agree that $i([\ty],[\ty'])=0$; otherwise, we take $i([\ty],[\ty'])=1+\ell$, where $\ell$ is the length of the intersection of the two paths joining $[\ty]$ and $[\ty']$ with their common root node.  

We may extend this notion to prime polynomials. If $F,G\in\P$, we define $i(F,G)$ as the index of coincidence of $[\ty_F]$ and $[\ty_G]$ as vertices of $\T$. For instance, in (\ref{sampletree}) we have $i(F_1,F_1)=i(F_1,F_2)=i(F_1,F_3)=i(F_1,F_4)=1$, $i(F_2,F_2)=i(F_2,F_3)=i(F_3,F_3)=4$, $i(F_2,F_4)=i(F_3,F_4)=2$, and $i(F_4,F_4)=6$.    

We say that a leaf of an OM representation of $f$ is \emph{isolated} if the previous node has only one branch. For instance, in (\ref{sampletree}) the leaf corresponding to $F_4$ is isolated and the other three leaves are not isolated.

\section{Computation of the genetics of a polynomial: the Montes algorithm}\label{secMontes}
In this section, we describe  the OM factorization algorithm developed by Montes in 1999, inspired by the ideas of Ore and MacLane \cite{montes}. It was first published in \cite{algorithm}, based on the theoretical background developed in \cite{HN}. 
In the context of this paper, the aim of the Montes algorithm is the computation of an OM representation of a given square-free polynomial $f\in\oo[x]$.

Let $\pset=\{F_1,\dots,F_t\}$ be the set of prime factors of $f$ in $\oo_v[x]$. 
For any type $\ty$ we denote $$\pset_\ty=\{F\in\pset\mid \ty\mid F\}\subset \pset.$$ Since $\ord_\ty(f)=\sum_{1\le j\le t}\ord_\ty(F_j)$, the set $\pset_\ty$ is empty if and only if $\ty\nmid f$. Also, if $\ord_\ty(f)=1$, then there is an index $j$ such that $\ord_\ty(F_j)=1$ and $\ord_\ty(F_k)=0$ for all $k\ne j$; thus, $\pset_\ty=\{F_j\}$ is a one-element subset in this case. 

The Montes algorithm is based on Theorem \ref{main}. The idea is to detect successive dissections of the set $\pset$ by subsets of the form
$\pset_\ty$ for adequate types. The first dissection is derived from the factorization $\overline{f}=\prod_\varphi \varphi^{\omega_\varphi}$ into the product of powers of pairwise different irreducible factors in $\F[y]$. Each irreducible factor $\varphi$ determines a type of order zero $\ty_\varphi=(\varphi)$ and the subset $\pset_{\ty_\varphi}$ contains all prime factors of $f$ whose reduction modulo $\m$ is a power of $\varphi$. By Hensel's lemma, we obtain a partition $\pset=\bigcup_\varphi\pset_{\ty_\varphi}$.    

In order to dissect $\pset_{\ty_\varphi}$, we choose a representative $\phi$ of $\ty_\varphi$; that is, a monic lifting of $\varphi$ to $\oo[x]$. By Lemma \ref{length2} and Corollary \ref{nextlength}, $\omega_\varphi=\ord_{\ty_\varphi}(f)$ is the length of the principal Newton polygon $N_{v_0,\phi}^-(f)$; thus, in order to compute this polygon we need only to compute the first $\omega_\varphi+1$ coefficients of the $\phi$-expansion of $f$. Then, for each slope $-\la$ of a side of $N_{v_0,\phi}^-(f)$ we compute the residual polynomial $R_{v_0,\phi,\la}(f)\in\F_1[y]=\F/(\varphi)[y]$. Finally, for each monic irreducible factor $\psi$ of $R_{v_0,\phi,\la}(f)$ in $\F_1[y]$ we consider the type of order one $\ty_{\la,\psi}=(\varphi;(\phi,\la,\psi))$. By definition, $\ord_{\ty_{\la,\psi}}(f)=\ord_\psi(R_{v_0,\phi,\la}(f))>0$,
so that the subsets $\pset_{\ty_{\la,\psi}}$ are not empty. By Theorem \ref{main}, $\pset_{\ty_\varphi}=\bigcup_{\la,\psi}\pset_{\ty_{\la,\psi}}$ is a partition. 

Each subset $\pset_{\ty_{\la,\psi}}$ is furtherly dissected by types obtained as enlargements of $\ty_{\la,\psi}$ with a similar procedure. By a certain 
process of \emph{refinement}, the algorithm is able to perform all these dissections dealing only with strongly optimal types. 

As mentioned above, when we reach a type $\ty$ with $\ord_\ty(f)=1$, then $\pset_\ty=\{F_j\}$ singles out a prime factor of $f$.

Let us briefly review the relevant subroutines which are used.\medskip

\noindent{\tt Factorization($\mathcal{F}$,\,$\varphi$)}

\noindent Factorization of $\varphi\in\mathcal{F}[y]$ into a product of irreducible polynomials in $\mathcal{F}[y]$.\medskip

\noindent{\tt Newton($\ty$,\,$\omega$,\,$g$)}

\noindent The type $\ty$ of order $i$ is equipped with a representative $\phi$. The routine computes the first $\omega+1$ coefficients $a_0,\dots,a_\omega$ of the canonical $\phi$-expansion $g=\sum_{0\le s}a_s\phi^s$, and the Newton polygon of the set of points $(s,v_i(a_s\phi^s))$ for $0\le s\le\omega$.\medskip

\noindent{\tt ResidualPolynomial($\ty$,\,$\la$,\,$g$)}

\noindent The type $\ty$ of order $i-1$ is equipped with a representative $\phi$. The routine computes the residual polynomial $R_{v_{i-1},\phi,\la}(g)\in\F_i[y]$.\medskip

\noindent{\tt Representative($\ty$)}

\noindent Computation of a representative of $\ty$ by the procedure described in section \ref{subsecConstruct}.

We now describe the Montes algorithm in pseudocode. 
Along the process of enlarging types by adding new level data, the order of a type $\ty$ is the largest level $i$ for which all three fundamental invariants $(\phi_i,\lambda_i,\psi_i)$ are assigned. We emphasize the type to which a certain level data belongs as a superindex: $\phi_i^\ty$, $\la_i^\ty$, $\psi_i^\ty$, etc.
\medskip

\noindent{\bf MONTES' ALGORITHM}\vskip 1mm

\noindent INPUT:

$-$ A discrete valued field $(K,v)$ with valuation ring $\oo$.

$-$ A monic square-free polynomial $f\in \oo[x]$. \medskip

\st{1}Initialize an empty list {\tt Forest}

\st{2}{\tt Factorization($\F$,$\overline{f}$)}

\st{3}FOR each monic irreducible factor $\varphi$ of $\overline{f}$ DO

\stst{4}Take a monic lifting $\phi\in \oo[x]$ of $\varphi$ and create a type $\ty$ of order zero with 
\vskip -.4mm \stst{}$\psi_0^\ty\leftarrow\varphi, \quad \omega_1^\ty\leftarrow \ord_\varphi \overline{f},\quad \phi_1^\ty\leftarrow \phi,\quad \F_1^\ty\leftarrow \F[y]/(\varphi)$

\stst{5}Initialize a tree of types $T_\varphi$ having $\ty$ as the unique vertex

\vskip -.4mm \stst{}Initialize a stack {\tt BranchNodes}$=[\ty]$ 

\stst{}{\bf WHILE $\#${\tt BranchNodes}\;$>0$  DO}

\ststst{6}Extract a type $\ty_0$ from {\tt BranchNodes}. Let $i-1$ be its order

\ststst{7}IF $\phi_i^{\ty_0}\mid f$ THEN $f\leftarrow f/\phi_i^{\ty_0}$ and add the leaf $(\ty_0;(\phi_i^{\ty_0},\infty,\hbox{--}))$ to $T_\varphi$


\ststst{8}$N\leftarrow$ {\tt Newton($\ty_0$,\,$\omega_i^{\ty_0}$,\,$f$)}

\ststst{9}FOR every side $S$ of $N$ DO

\stststst{10}$\lambda_i^{\ty_0}\leftarrow$ $-$slope of $S$, \ $R_i(f)\leftarrow$ {\tt ResidualPolynomial($\ty_0$,\,$\lambda_i^{\ty_0}$,\,$f$)}

\stststst{11}{\tt Factorization($\F_i^{\ty_0}$,\,$R_i(f)$)}

\stststst{12}FOR every monic irreducible factor $\psi$ of $R_i(f)$ DO

\ststststst{}(a) \ Set $\ty\leftarrow\ty_0$ and extend $\ty$ to an order $i$ type
by setting  

\vskip -.4mm \ststststst{}\qquad$\psi_i^{\ty}\leftarrow \psi,\quad \F_{i+1}^\ty\leftarrow\F_i^\ty[y]/(\psi)$

\ststststst{}(b) \ IF $\omega_i^{\ty_0}=1$ THEN add the leaf $\ty=\left(\ty_0;(\phi_i^\ty,\la_i^\ty,\psi_i^\ty)\right)$ to $T_\varphi$ 

\vskip -.4mm \ststststst{}\qquad and go to step {\bf 6} 

\ststststst{}(c) \ $\om_{i+1}^\ty\leftarrow \ord_{\psi}R_i(f)$, \ $\phi_{i+1}^\ty\leftarrow$ {\tt Representative($\ty$)}

\ststststst{}(d) \ IF $\deg \phi_{i+1}^\ty>\deg \phi_{i}^\ty$ THEN 

\vskip -.4mm \ststststst{}\qquad\quad\qquad add the vertex $\ty=\left(\ty_0;(\phi_i^\ty,\la_i^\ty,\psi_i^\ty)\right)$ to $T_\varphi$

\ststststst{}\qquad ELSE 
$\ \phi_i^{\ty}\leftarrow \phi_{i+1}^{\ty}, \quad \om_i^{\ty}\leftarrow\om_{i+1}^\ty$ 
\vskip -.4mm \ststststst{}\qquad\quad\qquad and delete the $(i+1)$-th level of $\ty$  

\ststststst{}(e) \ Add $\ty$ to {\tt BranchNodes}

\stst{}{\bf END WHILE}

\stst{13}Add the tree $T_\varphi$ to the list {\tt Forest}\medskip

\noindent OUTPUT:\vskip .1mm

$-$ The list {\tt Forest} of connected trees is an OM representation of $f$. 
  \medskip

The arguments of \cite{algorithm} show that the algorithm terminates and has the right output.
In that paper it was assumed that $K$ was a number field, but the arguments are valid for an arbitrary discrete valued field $(K,v)$. However, the design of the algorithm we present here has some changes with respect to the original design. Therefore, it may be worth clarifying some aspects on the flow and the output of the algorithm. 

Let $T$ be the output OM representation of $f$. The forest $T$ is the disjoint union of connected trees $T_\varphi$ attached to the different irreducible factors $\varphi$ of $\overline{f}$ in $\F[y]$.

\begin{remark}\label{rmkAlgo}
 
(1) \
An element in the list {\tt BranchNodes} is a vertex of $T_\varphi$, represented by a strongly optimal type $\ty_0$ of order $i-1$, together with attached data $\phi_i^{\ty_0}$ and $\omega_i^{\ty_0}$ at the $i$-th level. It may happen that different elements in {\tt BranchNodes} have the same underlying vertex $\ty_0$ of $T_\varphi$.

In step {\bf 12} we construct a type  $\ty_{\la,\psi}:=\ty=(\ty_0;(\phi,\la,\psi))$ of order $i$ and we compute $\omega_{\la,\psi}:=\omega_{i+1}^\ty=\ord_{\ty}(f)$ and a representative $\phi_{\la,\psi}:=\phi_{i+1}^\ty$. 
By Theorem \ref{main}, we have a partition $\pset_{\ty_0}=\bigcup_{\la,\psi}\pset_{\ty_{\la,\psi}}$.

If $\deg\phi_{\la,\psi}>\deg\phi$, then $\ty_{\la,\psi}$ yields a new vertex of $T_\varphi$ with previous node $\ty_0$. If $\deg\phi_{\la,\psi}=\deg\phi$, then $\ty_{\la,\psi}$ is not strongly optimal and it cannot be a vertex of $T_\varphi$. However, the subset $\pset_{\ty_{\la,\psi}}\subset\pset$  cannot be neglected. The algorithm adds to the list {\tt BranchNodes} the vertex corresponding to the type $\ty_0$ of order $i-1$ with data $\phi_{\la,\psi}$, $\omega_{\la,\psi}$ at the $i$-th level. This is called a \emph{refinement step}. In a future iteration of the WHILE loop the branches of this node will determine a partition of the old set $\pset_{\ty_{\la,\psi}}$ \cite[Sec. 3.2]{algorithm}. 

Note that a vertex $\ty_0$ may sprout some branches of $T_\varphi$ in a WHILE loop and then sprout some other branches in a future iteration of the WHILE loop, derived from a refinement step. These new branches of $\ty_0$ may again either lead to new vertices of $T_\varphi$ or to further refinement steps.\medskip

(2) \ In the original design of the algorithm in \cite{algorithm}, all leaves of $T$ were isolated, at the price of admitting leaves represented by types of order $r+2$, where $r$ is the Okutsu depth of the corresponding prime factor of $f$ \cite[Thm. 4.2]{okutsu}.

Since we want all leaves to have order $r+1$, we must admit non-isolated leaves. The algorithm stores a \emph{cutting slope} $\hcs$ as a ``secondary datum" of each type $\ty$ representing a leaf. This is a non-negative integer which vanishes if and only if the leaf is isolated.
The Newton polygon $N_{r+1}^-(f)$ determined by $\ty$ has a first side of slope $-\la_{r+1}<-\hcs$ whose end points have abscissas $0$ and $1$. All other sides of the polygon have slope greater than or equal to $-\hcs$ (see Figure \ref{figLastN}).

\begin{figure}[h]
\caption{Newton polygon $N_{r+1}^-(f)$ determined by a leaf of $T$. The line $L_{\cs}$ has slope $-\hcs$ and $f=\sum_{0\le s}a_s\phi_{r+1}^s$.}\label{figLastN}
\setlength{\unitlength}{5mm}
\begin{picture}(16,9.5)
\put(.85,7.8){$\bullet$}\put(2.85,3.8){$\bullet$}
\put(0,1.5){\line(1,0){6}}\put(1,.5){\line(0,1){8.5}}
\put(1,8){\line(1,-2){2}}\put(1.02,8){\line(1,-2){2}}
\multiput(3,1.4)(0,.25){11}{\vrule height2pt}
\multiput(.9,4)(.25,0){19}{\hbox to 2pt{\hrulefill }}
\put(2.1,6){\begin{footnotesize}$-\la_{r+1}$\end{footnotesize}}
\put(-1,7.9){\begin{footnotesize}$v_{r}(a_0)$\end{footnotesize}}
\put(-2.2,3.9){\begin{footnotesize}$v_r(a_1\phi_{r+1})$\end{footnotesize}}
\put(5.3,4.2){\begin{footnotesize}$L_{\cs}$\end{footnotesize}}
\put(2.9,.8){\begin{footnotesize}$1$\end{footnotesize}}
\put(.6,.8){\begin{footnotesize}$0$\end{footnotesize}}
\put(10.85,7.8){$\bullet$}\put(12.85,3.8){$\bullet$}
\put(10,1.5){\line(1,0){6}}\put(11,.5){\line(0,1){8.5}}
\put(11,8){\line(1,-2){2}}\put(11.02,8){\line(1,-2){2}}
\put(13,4){\line(2,-1){2}}\put(13.02,4){\line(2,-1){2}}
\multiput(11.5,5.4)(.1,-.1){35}{.}
\multiput(15,2.9)(.2,-.1){4}{.}
\multiput(13,1.4)(0,.25){11}{\vrule height2pt}
\multiput(10.9,4)(.25,0){9}{\hbox to 2pt{\hrulefill }}
\put(12.1,6){\begin{footnotesize}$-\la_{r+1}$\end{footnotesize}}
\put(9,7.9){\begin{footnotesize}$v_{r}(a_0)$\end{footnotesize}}
\put(7.8,3.9){\begin{footnotesize}$v_r(a_1\phi_{r+1})$\end{footnotesize}}
\put(15.1,1.8){\begin{footnotesize}$L_{\cs}$\end{footnotesize}}
\put(12.9,.8){\begin{footnotesize}$1$\end{footnotesize}}
\put(10.6,.8){\begin{footnotesize}$0$\end{footnotesize}}
\put(0.3,-0.4){\begin{footnotesize}isolated leaf ($h_{\cs}=0$)\end{footnotesize}}
\put(9.8,-0.4){\begin{footnotesize}non-isolated leaf ($h_{\cs}>0$)\end{footnotesize}}
\end{picture}
\end{figure}
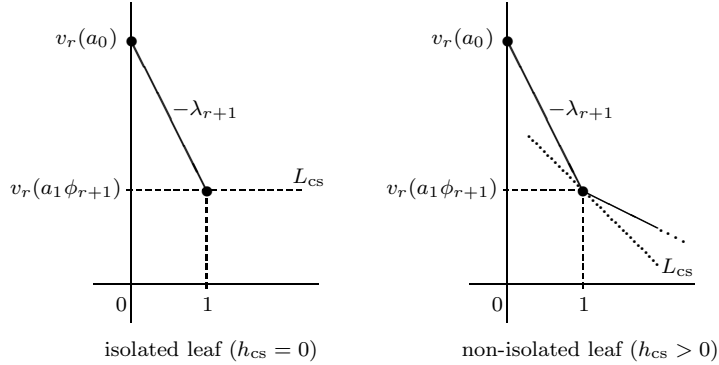

In \cite[Sec. 1.3]{newapp} a description may be found of some more secondary data stored in the types of an OM representation of $f$, which have been ignored in the pseudo-code description of the algorithm.\medskip

(3) \ For any type $\ty\in T$ the prime factors of $f$ in $\pset_\ty$ correspond to the leaves of $T$ for which $\ty$ is one of the vertices in the path joining the leaf with its root node.    
\end{remark}

The only algorithmic assumptions on the fields $K$ and $\F$ for the algorithm to work properly are the existence of efficient routines for the division with remainder of polynomials in $\oo[x]$ and the factorization of polynomials over finite extensions of the residue class field $\F$. 
The performance will depend as well on the efficiency of these two tasks. We have not yet analyzed the complexity of the algorithm in the general case, but for $\F$ a finite field,   the following complexity estimation was obtained in \cite[Thm. 5.14]{BNS}.     

\begin{theorem}
If $\F$ is a finite field, the complexity of the Montes algorithm, measured in number of operations in $\F$ is 
$$O\left(n^{2+\epsilon}+n^{1+\epsilon}(1+\delta)\log(q)+n^{1+\epsilon}\delta^{2+\epsilon}\right),$$
where $q=\#\F$, $n=\deg f$ and $\delta:=v(\dsc(f))$. 
\end{theorem}

\begin{example}\label{ex2}
Take $K=\Q$ and $v$ the $2$-adic valuation, so that $\F$ is the field with $2$ elements.  Let $z\in\overline{\F}$ be a generator of the field with $4$ elements. Consider the polynomial 
$$
\begin{array}{rl}
f=&x^{12} + 2x^{11} + 12x^{10} + 36x^9 + 100x^8 + 240x^7 + 544x^6+ 992x^5  
\\&+ 1328x^4 + 2080x^3 + 1728x^2+1600x+
    1125899906842816.
\end{array}
$$

The Montes algorithm computes the following OM representation of $f$: 
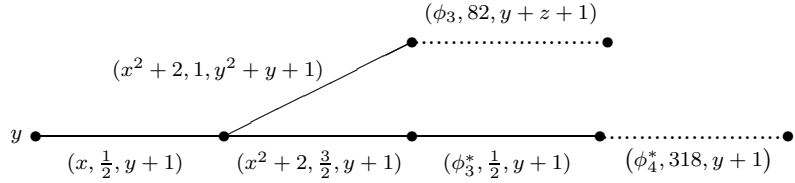
\begin{center}
\setlength{\unitlength}{5mm}
\begin{picture}(20,6)
\put(-.2,1.2){$\bullet$}\put(4.8,1.2){$\bullet$}\put(9.8,1.2){$\bullet$}\put(14.8,1.2){$\bullet$}\put(19.8,1.2){$\bullet$}
\put(-.7,1.3){\begin{footnotesize}$y$\end{footnotesize}}
\put(0,1.4){\line(1,0){5}}\put(10,1.4){\line(1,0){5}}
\put(5,1.4){\line(1,0){5}}
\put(5,1.4){\line(2,1){5}}
\put(.8,0.5){\begin{footnotesize}$(x,\frac12,y+1)$\end{footnotesize}}
\put(5.3,0.5){\begin{footnotesize}$(x^2+2,\frac32,y+1)$\end{footnotesize}}
\put(2,3){\begin{footnotesize}$(x^2+2,1,y^2+y+1)$\end{footnotesize}}
\put(10.3,4.5){\begin{footnotesize}$(\phi_3,82,y+z+1)$\end{footnotesize}}
\put(10.8,.5){\begin{footnotesize}$(\phi_3^*,\frac12,y+1)$\end{footnotesize}}
\put(9.8,3.7){$\bullet$}\put(15,3.7){$\bullet$}
\put(15.6,0.6){\begin{footnotesize}$\left(\phi^*_4,318,y+1\right)$\end{footnotesize}}
\multiput(10,3.85)(.2,0){25}{.}
\multiput(15,1.3)(.2,0){25}{.}
\end{picture}
\end{center}
where 
$$
\begin{array}{l}
\phi_3=x^4+2x^3+4x^2+4x+12,\\
\phi_3^*=x^4+4x^2+8x+4,\\
\phi_4^*=x^8+8x^6+16x^5+24x^4+96x^3+96x^2+128x+16.
\end{array}
$$

The polynomial $f$ has two prime factors in $\Z_2[x]$, say $f=FF^*$, with Okutsu depths $2$,\,3, respectively. From the structure of the tree we see that $i(F,F^*)=2$. 

The numerical MacLane-Okutsu invariants of $F$ are: 
$$
\begin{array}{llllll}
m_1=1,& e_1=2, & h_1=1, & f_1=1,& \la_1=1/2, & w_1=0, \\
m_2=2,& e_2=1, & h_2=1, & f_2=2, & \la_2=1,& w_2=1.
\end{array}
$$

The numerical MacLane-Okutsu invariants of $F^*$ are: 
$$
\begin{array}{llllll}
m^*_1=1,& e^*_1=2, & h^*_1=1, & f^*_1=1,& \la^*_1=1/2, & w^*_1=0, \\
m^*_2=2,& e^*_2=2, & h^*_2=3, & f^*_2=1,& \la^*_2=3/2, & w^*_2=1, \\
m^*_3=4,& e^*_3=2, & h^*_3=1, & f^*_3=1, & \la^*_3=1/2,& w^*_3=7/2.
\end{array}
$$

The formulas (\ref{Okinvariants}) allow us to compute Okutsu invariants of both factors from these data. For instance,
$$
\begin{array}{lllll}
e(F)=2,& f(F)=2, & \delta_0(F)=3, & \op{cap}(F)=2,  & \ind(F)=3,\\
e(F^*)=8,& f(F^*)=1, & \delta_0(F^*)=29/4, & \op{cap}(F^*)=47/8, & \ind(F^*)=20.
\end{array}
$$

Let us denote $\phi:=\phi_3$,  $\phi^*:=\phi^*_4$. We know that $f\approx\phi\phi^*$ is an OM factorization of $f$. The qualities of the approximations $\phi\approx F$, $\phi*\approx F^*$ are given by formula (\ref{quality}). The slopes of the last levels of the OM representations of $F$, $F^*$ are $\la_3=82$, $\la_4^*=318$, respectively. We obtain:
$$
\mu_{\infty,F}(\phi)=44,\qquad \mu_{\infty,F^*}(\phi^*)=47. 
$$
The estimation of Lemma \ref{measures} shows in both cases that the precision is at least $42$. 
\end{example}

\section{Algorithmic applications of polynomial genetics}\label{secComp}

We proceed to illustrate how to use the genetic data
to solve some typical problems related to  polynomials over local fields. The algorithms exploit the connection of some concrete problem  with the genetics of certain polynomials over local fields. This leads to an excellent practical performance. 


\subsection{Single-factor lifting and $v$-adic factorization}\label{subsecSFL}
Let $f\in\oo[x]$ be a monic square-free polynomial and let $f=F_1\cdots F_t$ be its factorization into a product of prime polynomials in $\oo_v[x]$.

A \emph{$v$-adic factorization} of $f$ is an approximate factorization with a prescribed precision; that is, a family of monic polynomials $P_1,\dots,P_t\in\oo[x]$ such that $P_j\equiv F_j \md{\m^\nu}$ for all $0\le j\le t$, for a prescribed positive integer $\nu$.

For many purposes, one needs sometimes to find an approximation with a prescribed quality to a single prime factor $F$ of $f$. This is the aim of the \emph{single-factor lifting algorithm} \cite{GNP}, abbreviated as SFL in what follows. The algorithm of \cite{GNP} was based on the original design of the Montes algorithm in which all trees of the output tree were isolated. Therefore, we review the design of SFL to adapt it to the present version of the Montes algorithm. 

The starting point of SFL is a leaf $\ty$ of an  OM representation of $f$ 
\begin{equation}\label{leaf}
\ty=(\psi_0;(\phi_1,\la_1,\psi_1);\dots;(\phi_r,\la_r,\psi_r);(\phi_{r+1},\la_{r+1},\psi_{r+1}))
\end{equation}
computed by the Montes algorithm. Let $F$ be the prime factor of $f$ singled out by $\ty$, and let $\t\in\kb$ be a root of $F$. We denote  $$V:=V_{r+1},\quad \phi:=\phi_{r+1},\quad  h_\phi:=\la_{r+1}=h_{r+1},\quad e:=e(F)=e_1\cdots e_r.
$$
The polynomial $\phi$ is a Montes approximation to $F$ as a factor of $f$. By (\ref{quality}), the quality of the approximation is:
$$
v(\phi(\t))=(V+h_\phi)/e=\delta_0(F)+h_\phi/e.
$$
The main loop of SFL computes a new Montes approximation $\Phi$ such that
$$
h_\Phi\ge 2h_\phi-\hcs.
$$
The Newton polygon $N_{v_r,\Phi}^-(f)$ coincides with $N_{v_r,\phi}^-(f)$ except for the side of largest slope (in absolute value) $-h_\Phi$, whose end points have abscissas $0$ and $1$ (see Figure \ref{figLastN}). In particular, the cutting slope $\hcs$ of $\ty$ separates again this initial side from the rest of the sides. Therefore, we may apply the SFL loop to $\Phi$ and iterate the procedure until we get a Montes approximation $\Phi$ with $h_\Phi$ large enough. By Lemma \ref{measures}, if $h_\Phi\ge e(\nu+\op{cap}(F)-\delta_0(F))$, then $\Phi \equiv F \md{\m^\nu}$.

After $k$ iterations of the SFL loop we get a Montes approximation $\Phi_k$ with
$$h_{\Phi_k}\ge h_\phi+(2^{k}-1)(h_\phi-\hcs).$$
Hence, for a given positive integer $H$, the number of iterations of the SFL loop that are needed to get $h_{\Phi_k}\ge H$ is $\lceil\log_2((H-\hcs)/(h_\phi-\hcs))\rceil$.

Let us briefly explain how to construct $\Phi$ from $\phi$. Consider the first two coefficients $a_0$, $a_1$ of the $\phi$-expansion of $f$:
$$
f=q\phi+a_0,\quad a_1= q\md{\phi}.
$$
A look at Figure \ref{figLastN} shows that $v_r(a_0)=v_r(a_1)+V+h_\phi$.
Let $\alpha\in\kb$ be a root of $\phi$ and let $K_\phi=K_v(\alpha)$, $\oo_\phi$ the valuation ring of $K_\phi$ and $\m_\phi$ the maximal ideal. Since $\deg a_0,\deg a_1<\deg\phi$, we have $\phi\nmid_{\mu_r}a_0$, $\phi\nmid_{\mu_r}a_1$, and Theorem \ref{vgt} shows that
$v(a_0(\alpha))=v_r(a_0)/e$, $v(a_1(\alpha))=v_r(a_1)/e$.

The following theorem is a slight variation of \cite[Thm. 5.1]{GNP}, where it was supposed that the leaf $\ty$ was isolated and $\hcs=0$.

\begin{theorem}\label{GNP5.1}
Let $a\in\oo[x]$ be a polynomial with $\deg a<\deg\phi$ and consider an integer $\hcs<h\le h_\phi$. Then, $\Phi:=\phi+a$ is a Montes approximation to $F$ with $h_\Phi\ge 2h-\hcs$ if and only if $a(\alpha)\equiv a_0(\alpha)/a_1(\alpha) \md{\m_\phi^{V+2h-\hcs}}$.
\end{theorem}

Let us show how to find a polynomial $a\in\oo[x]$ satisfying the condition of Theorem \ref{GNP5.1}. Compute a polynomial $\Psi\in K[x]$ with $\deg\Psi<n_F=\deg F=\deg \phi$ and $v_r(\Psi)=-v_r(a_1)$ \cite[Lem. 4.8]{GNP}. Multiply then,
$$
A_0:=a_0\Psi\md{\phi},\quad
A_1:=a_1\Psi\md{\phi}.
$$
Clearly, $a_0(\alpha)/a_1(\alpha)=A_0(\alpha)/A_1(\alpha)$, $v(A_0(\alpha))=(V+h_\phi)/e$ and $v(A_1(\alpha))=0$, so that $A_1(\alpha)$ is invertible in $\oo_\phi$. In order to compute $a\in\oo[x]$, it suffices to find an element $A_1^{-1}(\alpha)\in K_\phi$ with $A_1^{-1}(\alpha)A_1(\alpha)\equiv 1\mod (\m_\phi)^{h_\phi-\hcs}$ and then take $a(x)\in K[x]$ to be the unique polynomial of degree less than $n_F$ satisfying $a(\alpha)=A_0(\alpha)A_1^{-1}(\alpha)$. By the formulas in (\ref{Okinvariants}) we have $v(a(\alpha))>\exp(F)=\exp(\phi)$, so that $a(x)\in\oo[x]$.

In order to avoid inversions in $K_\phi$, we may compute the  approximation $A_1^{-1}(\alpha)$ to $A_1(\alpha)^{-1}$ by the classical Newton iteration:
$$
x_{k+1}=x_k(2-A_1(\alpha)x_k),
$$
starting with a lift $x_0\in\oo_\phi$ of the inverse of $A_1(\alpha)+\m_\phi$ in the residue field $\oo_\phi/\m_\phi$. In \cite[Sec. 4.2]{newapp} it is explained how to compute $x_0$.\medskip

\noindent{\bf SINGLE-FACTOR LIFTING}\vskip 1mm

\noindent INPUT:

$-$ A discrete valued field $(K,v)$ with valuation ring $\oo$.

$-$ A monic square-free polynomial $f\in \oo[x]$.

$-$ A leaf $\ty$ of order $r+1$, as in (\ref{leaf}), of an OM representation of $f$.

$-$ A positive integer $H$. \medskip

\st{1} $\phi\gets\phi_{r+1}$,\quad$q, a_0 \gets \op{quotrem}(f,\phi)$,$\quad a_1 \gets q \md \phi$

\st{2} $h_\phi\gets v_r(a_0)-v_r(a_1\phi)$

\st{3} Find $\Psi\in K[x]$ with $\deg\Psi<\deg\phi$ and $v_r(\Psi)=-v_r(a_1)$ \quad\cite[Lem. 4.8]{GNP}

\st{4} $A_0\gets \Psi a_0\md\phi$,\quad$A_1\gets\Psi a_1\md\phi$

\st{5} Find $A_1^{-1}\in \oo[x]$ with $A_1^{-1}(\alpha)A_1(\alpha)\equiv 1\md{\m_\phi}$ \quad\cite[Sec. 4.2]{newapp}

\st{6} FOR $i=1$ TO $\lceil\log_2(h_\phi-\hcs)\rceil$ DO

\stst{}$A_1^{-1}\gets A_1^{-1}(2-A_1 A_1^{-1})\md\phi$

\st{7} $a\gets A_0A_1^{-1}\md\phi$,\quad$\Phi\gets\phi+a$

\st{8} FOR  $i=1$ TO $\lceil\log_2((H-\hcs)/(h_\phi-\hcs))\rceil-1$ DO

\stst{} (a) \ $q,a_0 \gets \op{quotrem}(f,\Phi),\quad a_1 \gets q \md \Phi$

\stst{} (b) \ $A_0\gets \Psi a_0\md\Phi,\quad A_1\gets\Psi a_1\md\Phi$

\stst{} (c) \ $A_1^{-1}\gets A_1^{-1}(2-A_1 A_1^{-1})\md\Phi$

\stst{} (d) \ $a\gets A_0A_1^{-1} \md\Phi,\quad \Phi\gets\Phi+a$
\medskip

\noindent OUTPUT:\vskip .1mm

$-$ A Montes approximation $\Phi$ to the prime factor $F$ of $f$ attached to $\ty$, such that $h_\Phi\ge H$.
\medskip

Note that step {\bf 7} terminates a first iteration of the SFL loop. The rest of iterations are performed by the loop described in step {\bf 8}. For these iterations it is not necessary to start over the inversion loop of step {\bf 6}. In fact, let $\alpha_k\in \kb$ be a root of $\Phi_k$ and denote by $A_{1,k}$ the $k$-th polynomial $A_1$. Then, for $\hcs<h\le h_{\Phi_k}$, the inversion of $A_{1,k}(\alpha_k)$ modulo $\m_{\Phi_k}^{h-\hcs}$ is also an inversion of  $A_{1,k+1}(\alpha_{k+1})$ modulo $\m_{\Phi_{k+1}}^{h-\hcs}$ \cite[Prop. 5.5]{GNP}; hence, we get the desired inversion of  $A_{1,k+1}(\alpha_{k+1})$ modulo $\m_{\Phi_{k+1}}^{2h-\hcs}$ just by one iteration of the Newton inversion procedure in step {\bf 8}(c).

The complexity of the SFL routine was analyzed in \cite[Lem. 6.5]{GNP} and \cite[Thm. 5.16]{BNS}. In the next result we denote $n=\deg f$, $n_F=\deg F$ and $\delta_F=v(\dsc(F))$.

\begin{theorem}\label{complexitySFL}
The SFL routine requires $O\left(nn_F\nu^{1+\ep}+n\delta_F^{1+\ep}\right)$ operations in $\F$ to compute a Montes approximation $\Phi$ to $F$ as a factor of $f$, with precision $\nu$.
\end{theorem}

By applying the SFL routine to each leaf of an OM representation of $f$, we get an OM factorization $f\approx P_1\cdots P_t$ such that $P_j\equiv F_j\md{\m^\nu}$ for all $j$.

\begin{theorem}\label{factorization}
If $\F$ is a finite field, a combined application of the Montes and SFL algorithms, computes an OM factorization of $f$  with precision $\nu$, at the cost of
\[O\left(n^{2+\epsilon}+n^{1+\epsilon}(1+\delta)\log q+n^{1+\epsilon}\delta^{2+\epsilon}+n^2\nu^{1+\epsilon}\right)\]
operations in $\F$.
\end{theorem}

\begin{example}\label{ex3}
Recall the OM factorization $f\approx \phi\phi^*$ of Example \ref{ex2}. A single iteration of the main loop of SFL for each approximation of the true factors $f=FF^*$ yields the following improvements:
$$
\begin{array}{rl}
\Phi=&\!\!\!x^4\! -\! 869643860553342248938373118x^3 + 895292343076575293699260420x^2\\
& - 358277240246400736326320124x -
    615563580557575482075250676\\ \\
\Phi^*=&\!\!\! x^8\! -\! 368296178732038025960751104x^7 - 158699985612499241777758200x^6\\
& + 440432535828627390937956368x^5 -  70934084478318519720607720x^4\\
& + 468084806048993171281543264x^3 + 345452998984777616876109920x^2\\
&-
    244862856588991367554793344x - 417188598541852473806553072.
\end{array}
$$
The qualities of these new approximations are:
$$
\mu_{\infty,F}(\Phi)=85, \qquad \mu_{\infty,F^*}(\Phi^*)=88.
$$
By Lemma \ref{measures}, the precision is at least $83$ in both cases. 
\end{example}

\subsection{Computation of the pseudo-valuation $\mu_{\infty,F}$}\label{subsecValue}
Let $F$ be a prime factor in $\oo_v[x]$ of an irreducible polynomial $f\in\oo[x]$. Let $\t\in\kb$ be a root of $F$.

We present an algorithm for the computation of $\mu_{\infty,F}(g)=v(g(\t))$ for a given polynomial $g\in\oo[x]$. The basic idea is that this value should be deduced from a comparison of the genomic trees of $F$ and $g$. More precisely, if we find an inductive valuation $\mu$ and a key polynomial $\phi$ for $\mu$ such that $\phi\mmu F$ and $\phi\nmid_\mu g$, then
$\mu_{\infty,F}(g)=\mu(g)$ by Theorem \ref{vgt}. From a computational perspective this amounts to finding a type $\ty$ such that $\ty\mid F$, $\ty\nmid g$, leading to $\mu_{\infty,F}(g)=\mu_\ty(g)$.

Let $r$ be the Okutsu depth of $F$ and let $\ty$ be the leaf of an OM representation of $f$, as in (\ref{leaf}), corresponding to $F$. Since $\ty\mid F$, we may check if $\op{Trunc}_i(\ty)\nmid g$ holds for some $0\le i\le r+1$, leading to $v(g(\t))=\mu_i(g)$. This fails if $\ty\mid g$ (for instance, if $\phi_{r+1}\mid g$). In this case, we improve the Okutsu approximation $\phi_{r+1}$ by applying one loop of the SFL routine; then, we replace the $(r+1)$-th level of $\ty$ by the data $(\phi,\la,\psi)$ determined by the new choice of $\phi_{r+1}$, and we test again if the new $\ty$ divides $g$.

If $\ty\mid g$, then $\phi_{r+1}$ is simultaneously close to a prime factor of $f$ and to a prime factor of $g$; hence, if $f$ and $g$ do not have a common prime factor in $\oo_v[x]$, after a finite number of steps the renewed type $\ty$ will not divide $g$. On the other hand, if $f$ and $g$ have a common prime factor, they must have a common irreducible factor in $\oo[x]$ too; since $f$ is irreducible, necessarily $f$ divides $g$ and $g(\t)=0$.
\medskip

\noindent{\bf $v$-VALUE ROUTINE}\vskip 1mm

\noindent INPUT:

$-$ A discrete valued field $(K,v)$ with valuation ring $\oo$.

$-$ A monic irreducible polynomial $f\in \oo[x]$.

$-$ A leaf $\ty$ of order $r+1$, as in (\ref{leaf}), of an OM representation of $f$.

$-$ A polynomial $g\in \oo[x]$.
\medskip

\st{1} $g\gets g\md f$

\st{2} IF $g = 0$ THEN  RETURN $\infty\ $ ELSE  $\ \nu \gets v_0(g),\quad g\gets g/\pi^\nu$

\st{3} IF $\psi_0\nmid \overline{g}$ THEN RETURN $\nu$

\st{4} FOR $i=1$ to $r+1$ DO

\stst{} (a) \ Compute $N_i^-(g)$ and the left end point $(s,u)$ of $S_{\la_i}(g)$ (section \ref{subsecNewton})

\stst{} (b) \ Compute $R_i(g)$

\stst{} (c) \ IF $\psi_i\nmid R_i(g)$ then RETURN $\nu+(u+s\la_i)/e_1\cdots e_{i-1}$

\st{5} WHILE $\psi_{r+1}\mid R_{r+1}(g)$ DO

\stst{} (a) \ Apply one loop of SFL to improve $\phi_{r+1}$

\stst{} (b) \ Set $\la_{r+1}$ as the largest slope in absolute value of the new $N_{r+1}^-(f)$

\stst{} (c) \ $\psi_{r+1}\gets R_{r+1}(f)$

\stst{} (d) \ Compute the new $N_{r+1}^-(g)$ and the left end point $(s,u)$ of $S_{\la_{r+1}}(g)$

\stst{} (e) \ Compute $R_{r+1}(g)$

\st{6} RETURN $\nu+(u+s\la_{r+1})/e(F)$\medskip

\noindent OUTPUT:\vskip .1mm

$-$ $v(g(\t))$, where $\t\in\kb$ is a root of the prime factor $F$ of $f$ attached to $\ty$.
\medskip

In step {\bf 4}(c) we use $v(g(\t))=\mu_i(g)=(u+s\la_i)/e_1\cdots e_{i-1}$, by Lemma \ref{muprim}.

Let $\alpha\in \overline{K}$ be a root of $f$ and $L=K(\alpha)$ the finite extension of $K$ generated by $\alpha$. Every $\beta\in L$ is of the form $\beta=g(\alpha)$ for some $g\in K[x]$; hence, the $v$-routine computes
$v(\iota(\beta))$, where $\iota\colon L\to \kb$ is the embedding determined by $\alpha\mapsto \t$.

Suppose $K$ is a number field and $v=v_\p$ is the $\p$-adic valuation attached to a prime ideal $\p$ of $K$. Then, the prime factor $F$ of $f$ corresponds to a prime ideal $\mathfrak{P}$ of $L$ dividing $\p$. As explained in \cite{newapp}, the $v$-routine may be used to compute the $\mathfrak{P}$-adic valuation mapping $v_\mathfrak{P}\colon L\to \Z$ by the formula
$$
v_\mathfrak{P}(\beta)=e(\mathfrak{P}/\p)v(\iota(\beta))=e(F)v(\iota(\beta)).
$$

Similarly, the $v$-routine may be used to compute the order of a function at a given point
on an algebraic curve.

\subsection{Index of a square-free polynomial}{Index of a square-free polynomial}
Let $F\in\P$ be a prime polynomial and $\t\in\kb$ a root of $F$. The Dedekind domain $\oo_F$ is a free $\oo_v$-module of rank $n_F= \deg F$. Since $ \oo_v[\t]\subset \oo_F$ is also a free $\oo_v$-module of the same rank, the quotient $\oo_F/\oo_v[\t]$ is an $\oo_v$-module of finite length. This length is the \emph{index} of $F$ and we denote it by $\ind(F)$. 

Now, let $f$ be a square-free polynomial in $\oo[x]$ with prime factorization $f=F_1\cdots F_t$
in $\oo_v[x]$. We define the \emph{index} of $f$ as
\begin{equation}\label{index}
\ind(f):=\sum_{1\le i\le t}\ind(F_i)+\sum_{0\le i<j\le t}v(\res(F_i,F_j)).
\end{equation}

By using a formula of \cite[Thm. 4.18]{HN}, in \cite{algorithm} it was shown how to compute  $\ind(f)$ as the accumulation of the number of points of integer coordinates lying below all Newton polygons that occur along the flow of the Montes algorithm.

Alternatively, we may compute $\ind(f)$ by a closed formula in terms of the data stored in a OM-representation of $f$. By either method, we obtain the value of $\ind(f)$ as a by-product of the Montes algorithm at a negligible cost.

\begin{proposition}\label{indres}
Let $F,G\in\P$ be two prime factors of $f$ and let $i=i(F,G)$ be their index of coincidence. Then,
$$
\begin{array}{l}
 \ind(F)=\deg F\left(\op{cap}(F)-1+e(F)^{-1}\right)/2,\\
v(\res(F,G))=\deg F\deg G\left(V_i+\min\{\la_F^G,\la_G^F\}\right)/e_1\cdots e_{i-1}m_i.
\end{array}
$$
\end{proposition}

The capacity $\op{cap}(F)$ was given in (\ref{Okinvariants}) and the index of coincidence $i(F,G)$ was defined in section \ref{subsecGenomic}. The types $\ty_F$ and $\ty_G$ coincide at the levels $0,1,\dots,i-1$ and the data $e_1,\dots,e_{i-1},m_i,V_i$ are common to both types. The rational numbers $\la_F^G$, $\la_G^F$ are the {\em hidden slopes} of the pair $F,G$ and they are stored as secondary data of the types of the OM representation of $f$. They are obtained in the first iteration of the WHILE loop of the Montes algorithm where the prime factors $F$, $G$ are separated by the branching process. Their name reflects the fact that they cannot be read in the genomic trees if one of the branches of $F$ or $G$ detected in that WHILE loop led to a refinement step.

The formula for $\ind(F)$ was proved in \cite[Prop. 3.5]{GNP}.
Since $F$ and $G$ are irreducible in $\oo_v[x]$, we have $v(\res(F,G))=\deg Fv(G(\t))$, where $\t$ is a root of $F$. Hence, the formula for $v(\res(F,G))$ may be derived from  \cite[Prop. 4.7]{newapp}.

\begin{example}\label{ex4}
Let  $f=FF^*$ be the  polynomial of Example \ref{ex2}. We have already seen that $i(F,F^*)=2$, $\ind(F)=3$ and $\ind(F^*)=20$.
Along the application of the Montes algorithm to $f\in \Q[x]$ with respect to the $2$-adic valuation, it occurs no refinement. Hence, the hidden slopes are $\la_F^{F^*}=\la_2=1$, $\la_{F^*}^F=\la^*_2=3/2$, so that $v(\res(F,F^*))=24$ by Proposition \ref{indres}. After (\ref{index}), we obtain $\ind(f)=47$. 
\end{example}

If $f$ is irreducible and separable in $\oo[x]$, then the non-negative integer $\ind(f)$ defined in (\ref{index}) coincides with the usual concept of index. In fact, let $\alpha\in\overline{K}$ be a root of $f$ and let $B$ be the integral closure of $\oo$ in the finite extension $L=K(\alpha)$. Since the extension $L/K$ is separable and $\oo$ is a PID, the Dedekind domain $B$ is a free $\oo$-module of rank $n=\deg f$. From the well-known identities relating indices and discriminants \cite[III, \S2-4]{serre}
we deduce that:
$$
\ind(f)=v\left(\left(B\colon\oo[\alpha]\right)\right)=\op{length}_\oo(B/\oo[\alpha]),
$$
where the \emph{index} $\left(B\colon\oo[\alpha]\right)$ is the $\oo$-ideal defined in
\cite[I, \S5]{serre}.

\subsection{Computation of discriminants and resultants}
Let $g,h\in\oo[x]$ be two monic polynomials having no common prime factors. In \cite[Thm. 4.10]{HN} a formula for $v(\res(g,h))$ was obtained in terms of the intersection of two OM re\-presentations of $g$ and $h$. In \cite{Ndiff} a concrete algorithm was designed to carry out this computation. This yields a fast computation of  $v(\res(g,h))$ and/or $v(\dsc(g))$ in cases where the naive computation of $\res(g,h)$ or $\dsc(g)$ is unfeasible because the polynomials have large degree or large coefficients.  

\subsection{Computation of $v$-integral bases}
Suppose that $f\in\oo[x]$ is monic and irreducible. Let $\alpha\in \overline{K}$ be a root of $f$ and $L=K(\alpha)$ the finite extension of $K$ generated by $\alpha$.
The integral closure $B$ of $\oo$ in $L$ is a Dedekind domain.

We suppose that $B$ is finitely generated as an $\oo$-module. This condition holds under very natural assumptions; for instance, if $L/K$ is separable, or $(K,v)$ is complete, or $\oo$ is a finitely generated algebra over a field \cite[I, \S4]{serre}.

Under this assumption, $B$ is a free $\oo$-module of rank $n=\deg f$. A \emph{$v$-integral basis} of $L$ is by definition an $\oo$-basis of $B$.

The \emph{method of the quotients} is valid in this general setting and it computes a $v$-integral basis of $L$ as a by-product of a standard application of the Montes algorithm to $f$ \cite{bases}. Whenever a $\phi$-expansion of $f$ is computed along the flow of the algorithm, the quotients of the successive divisions with remainder are stored. The members of the $v$-integral basis are computed as $g_i(\alpha)\pi^{-d_i}$, for $0\le i<n$, where $g_i\in\oo[x]$ are an adequate product of the stored quotients, and the non-negative integers $d_i$ are determined by combinatorial data of the Newton polygons considered by the algorithm.

The complexity of the method is the cost of the Montes algorithm plus $O(n)$ multiplications in $\oo[\alpha]$. This method is extremely fast in practice.

\subsection{OM representations of prime ideals in a number field}\label{subsecNF}
Suppose $K=\Q$ and $v=v_p$ is the $p$-adic valuation attached to a prime number $p$, so that $\oo$ is the local ring $\Z_{(p)}$.

Let $f\in\Z[x]$ be a monic irreducible polynomial, $\alpha\in\overline{\Q}$ a root of $f$ and $L=\Q(\alpha)$ the number field determined by $f$. Denote by $\Z_L$ the ring of integers of $L$.

After a celebrated theorem by Hensel, the prime ideals of $L$ lying over $p$ are in 1-1 correspondence with the prime factors of $f$ in $\Z_p[x]$. From a computational perspective, the prime ideals may be represented in a computer as the OM representations of these prime factors of $f$, which may be computed by a single application of the Montes algorithm  with input $f,v_p$. The main arithmetic tasks concerning prime ideals may be easily performed by using the MacLane-Okutsu invariants and the operators encoded by these OM representations.

In \cite{newapp} it was presented a computational approach to ideal theory in number fields based on this principle. This approach has the advantage that many arithmetic tasks may be carried out avoiding the computation of the maximal order of $L$ and the factorization of the discriminant of the defining polynomial $f$. The most relevant ones are:

\begin{enumerate}
\item Compute the $\p$-adic valuation $v_\p\colon L^*\to \Z$, for any prime ideal $\p$ of $L$.
\item Obtain the prime ideal decomposition of a fractional ideal.
\item Compute a two-elements representation of a fractional ideal.
\item Add, multiply and intersect fractional ideals.
\item Compute the reduction maps $\Z_L\to \Z_L/\p^a$.
\item Solve Chinese remainders problems.
\item Compute a $p$-integral basis of $L$.
\end{enumerate}

The genetic-based routines designed to perform these tasks are much faster than the classical ones, especially for those number fields defined by polynomials with a large genomic tree at some prime.

The Magma package \emph{+Ideals.m} based on this approach may be downloaded from
 {\tt http://www-ma4.upc.edu/$\sim$guardia/+Ideals.html}.

\subsection{OM representations of places in a function field}\label{subsecFF}
Suppose that $K=k(t)$ is the function field of the projective line $\P^1$ over a field $k$.
An irreducible polynomial $f(t,x)\in k[t,x]$, separable over $k[t]$, determines a unique smooth projective curve $C$ as the normalization of the projective closure of the affine curve $f(t,x)=0$. The function field of $C$ is $L=k(t,x)=K[x]/(f)$.

Let $A=k[t]$ and denote by $A_\infty=k[t^{-1}]$ the local ring at the point at infinity of $\P^1$. Let $B$ and $B_\infty$ be the integral closures in $L$ of $A$ and $A_\infty$, respectively. As indicated in \cite{hess}, a divisor of $C$ may be identified to a  pair $D=(I, I_\infty)$ of fractional ideals of $B$ and $B_\infty$ respectively. The Riemann-Roch space attached to the divisor is simply $L(D)= I\cap I_\infty$.

In this representation, a prime divisor corresponds to a prime ideal in either $B$ or $B_\infty$, and it may be represented by a prime factor of $f(t,x)$ over the completion of $A$ or $A_\infty$ at a place (finite or infinite) of $K$.

J.-D. Bauch has developed a Magma package \emph{+Divisors.m} where divisors are manipulated as such pairs $D=(I, I_\infty)$, and ideals are handled  as OM representations. 
This approach leads to fast OM routines to compute the genus of a curve \cite{B} and the divisor of a function. Also, it yields an acceleration of the classical methods to compute $k$-bases of the Riemann-Roch spaces of divisors defined over $k$.

\subsection{Construction of field extensions with prescribed ramification}\label{subsecPrescribed}
The  procedure for the  construction of types described in section \ref{subsecConstruct} can be used to generate number fields $L$ with prescribed decomposition of several rational primes. For every prescribed factorization $p\Z_L=\p_1\dots\p_g$ one constructs proper types $\ty_1,\dots,\ty_g$; the product of  the $\phi$-polynomials in the last level of all these types plus a high enough power of $p$ will generate a number field where $p$ has the desired factorization. Actually, we can even prescribe different values for the Okutsu invariants of the prime factors of the generating polynomial corresponding to the different prime ideals.  In order to combine prescribed data for different prime numbers, one has only to apply the Chinese remainder theorem. The same idea works for function fields.

We have used these ideas in \cite{GNP} to design a bank of benchmark polynomials for the analysis of  algorithms on number fields or function fields.





\begin{thebibliography}{}
\bibitem{B}J.-D. Bauch, \emph{Genus computation of global function fields},	arXiv:1209.0309v3 [math.NT], Journal of Symbolic Computation,  to appear.


\bibitem{BNS}
J.-D. Bauch, E. Nart, H. D. Stainsby, \emph{Complexity of OM factorizations of polynomials over local fields}, LMS Journal of Computation and Mathematics {\bf 16} (2013), 139--171.

\bibitem{Rideals}
J. Fern\'andez, J. Gu\`{a}rdia, J.  Montes, E.  Nart, \emph{Residual ideals of MacLane valuations}, arXiv:1305.0775v2 [math.NT].

\bibitem{okutsu}
J. Gu\`{a}rdia, J.  Montes, E.  Nart, \emph{Okutsu invariants and Newton polygons}, Acta Arithmetica {\bf145} (2010), 83--108.

\bibitem{algorithm}
J. Gu\`{a}rdia, J.  Montes, E.  Nart, \emph{Higher  Newton polygons in the computation of discriminants and prime ideal decomposition in number fields},
Journal de Th\'eorie des Nombres de Bordeaux {\bf 23} (2011), no. 3, 667--696.

\bibitem{HN}
J. Gu\`{a}rdia, J. Montes, E. Nart, \emph{Newton polygons of higher order in algebraic number theory}, Trans. Amer. Math. Soc.  {\bf 364} (2012), no. 1, 361--416.

\bibitem{newapp}
J. Gu\`{a}rdia, J. Montes, E.  Nart, \emph{A new computational approach to ideal theory in number fields}, Foundations of Computational Mathematics {\bf 13} (2013), 729--762.

\bibitem{bases} J. Gu\`ardia, J. Montes, E. Nart, \emph{Higher Newton polygons and integral bases}, arXiv: 0902.3428v2[math.NT].

\bibitem{GNP}J. Gu\`{a}rdia, E.  Nart, S. Pauli, \emph{Single-factor lifting and factorization of polynomials over local fields}, Journal of Symbolic Computation {\bf 47} (2012), 1318--1346.

\bibitem{hess} F. Hess, \emph{Computing Riemann-Roch spaces in algebraic function fields and related topics}, J. Symb. Comput. {\bf 33} (2002), 425--445.



\bibitem{mcla} S. MacLane, \emph{A construction for absolute values in polynomial rings}, Transactions of the American Mathematical Society, {\bf40}(1936), pp. 363--395.

\bibitem{mclb} S. MacLane, \emph{A construction for prime ideals as absolute values of an algebraic field}, Duke Mathematical Journal {\bf2}(1936), pp. 492--510.

\bibitem{montes}J. Montes, \emph{Pol\'{\i}gonos de Newton de orden superior y aplicaciones aritm\'eticas}, PhD Thesis, Universitat de Barcelona, 1999.




\bibitem{Ndiff} E. Nart, \emph{Local computation of differents and discriminants}, Mathematics of Computation 83 (2014), no 287, 1513--1534.

\bibitem{Ok}
K. Okutsu, \emph{Construction of integral basis, I, II}, Proceedings of the Japan Academy {\bf 58}, Ser. A (1982), 47--49, 87--89.






\bibitem{serre} J. P. Serre, \emph{Corps Locaux}, second Edition, Hermann, Paris, 1968.

\bibitem{Vaq}
M. Vaqui\'e, \emph{Extension d'une valuation}, Transactions of the American Mathematical Society  {\bf 359} (2007), no. 7, 3439--3481.

\bibitem{Vaq2}M. Vaqui\'e, \emph{Famille admissible de valuations et d\'efaut d'une extension}, Journal of Algebra {\bf 311} (2007), no. 2, 859--876.

\bibitem{Vaq3}M. Vaqui\'e, \emph{Extensions de valuation et polygone de Newton}, Annales de l'Institute Fourier (Grenoble) {\bf 58} (2008), no. 7, 2503--2541.

\end{thebibliography}
\end{document}